\documentclass[letterpaper, 10 pt, conference]{ieeeconf}  
\pdfoutput=1                                              

\IEEEoverridecommandlockouts                              
\newif\ifdraft
\drafttrue 

\title{\LARGE \bf
A Parallel Dual Fast Gradient Method for MPC
Applications{$^*$}}

\author{\centering Laura Ferranti{$^{1}$},~Tam\'as Keviczky{$^{1}$
\thanks{*This research is supported by the European Union's Seventh Framework Programme FP7/2007-2013 under grant agreement n$\,^\circ$ AAT-2012-RTD-2314544 entitled ``Reconfiguration of Control in Flight for Integral Global Upset Recovery (RECONFIGURE)''.}
 \thanks{{$^{1}$}L. Ferranti and T.
Keviczky are with the Delft Center for Systems and Control, Delft University of Technology, Delft, 2628 CD, The Netherlands,
        {\tt\small $\{$l.ferranti,t.keviczky$\}$@tudelft.nl}}}%
}

\usepackage{amssymb}          
\usepackage{amsmath}          
\usepackage{amsthm}
\usepackage{graphicx}         
\usepackage{url}
\usepackage{color}
\usepackage{enumerate}
\usepackage{multirow}
\usepackage[font= footnotesize]{caption}
\usepackage{verbatim}
\usepackage{graphics}
\usepackage{algpseudocode}
\usepackage{algorithm}
\usepackage{algorithmicx}
\usepackage{setspace}

\newtheorem{assumption}{Assumption}
 \theoremstyle{remark}
 \newtheorem{rem}{Remark}
 \theoremstyle{remark}
  \newtheorem{lem}{Lemma} 
  \theoremstyle{plain}
   \newtheorem{thm}{Theorem}
\theoremstyle{plain}

\DeclareMathOperator{\1}{\mathbf{1}}
\DeclareMathOperator{\ubf}{\mathbf{u}}
\DeclareMathOperator{\mcU}{\mathcal{U}}
\newcommand\upbound[1]{\ensuremath{\min\limits_{j=1,\ldots,p_{#1}}\!\!\!\{-(G_{#1}\tilde y_{#1}+g_{#1})_j\}}}
\newcommand\upboundspecial[1]{\ensuremath{\!\!\!\!\!-\!\frac{\min\limits_{j=1,\ldots,p_{#1}}\{-(G_{#1}\tilde
y_{#1}+g_{#1})_j\}}{2^t(1+2\max_{C_{#1})}}}}
\newcommand\maxC[1]{\ensuremath{\max\limits_{j=1,\ldots,p_t}\!\!\!\big\{\sum_{i=1}^n|[C_{#1}]_{j,i}|\big\}}}
\newcommand\sigmamax[1]{\ensuremath{|A|^{#1}}} \newcommand\timeint[2]{\ensuremath{t\!=\!#1,\ldots,#2}}
\newcommand\posdef[1]{\ensuremath{\mathbb S^{{#1}}_{>0}}}
\newcommand\possemdef[1]{\ensuremath{\mathbb S^{{#1}}_{\geq 0 }}}
\newcommand\bfV{\ensuremath{\mathbf V}}
\newcommand{\itglob}{\textit{consolidated}~}
\newcommand{\itloc}{\textit{local}~}

\begin{document} 
\maketitle
\thispagestyle{empty}
\pagestyle{empty}
\hyphenation{pro-blem}
\hyphenation{sub-pro-blem}
\hyphenation{sub-pro-blems}
\hyphenation{sol-ving}
\hyphenation{pa-ral-le-li-za-ti-on}
\hyphenation{tigh-te-ned}
\hyphenation{nu-me-ri-cal}
\begin{abstract}  
We propose a parallel adaptive constraint-tightening
 approach to solve a linear model predictive control problem for discrete-time
 systems, based on inexact numerical optimization algorithms and operator
 splitting methods. The underlying algorithm first splits the original problem in as
 many independent subproblems as the length of the prediction horizon. Then, our
 algorithm computes a solution for these subproblems in parallel by exploiting auxiliary 
 tightened subproblems in order to certify the control law in terms of suboptimality and 
 recursive feasibility, along with closed-loop stability of the 
 controlled system. Compared to prior approaches based on constraint tightening,
  our algorithm computes the tightening parameter for each subproblem 
  to handle the propagation of errors introduced by the parallelization of the
  original problem. Our simulations show the computational benefits of the
  parallelization with positive impacts on performance and numerical conditioning when compared
   with a recent nonparallel adaptive tightening scheme.
\end{abstract}
\section{Introduction}
Model Predictive Control (MPC) is a consolidated control technique that can
efficiently handle constraints on the process to be controlled. 
Nevertheless, its application is not yet widespread in many domains where
real-time computational constraints and requirements of certified solutions are
of major concern, such as aerospace or automotive applications. There is a
growing interest in both industry and academia for exploring parallel solutions
to MPC problems~(\cite{PQP2,KogelIFAC,StathopoulosECC13}), especially in light
of the emerging many-core architectures, aiming to improve the computational
efficiency of solving the underlying optimization problem.

\textbf{Contribution}. In this paper, we explore the use of
parallelization techniques to efficiently solve a typical MPC problem
for a linear discrete-time system, with a substantial computational speedup 
compared to nonparallel implementations.
Our proposed algorithm combines the use of Alternating Direction
Method of Multipliers (ADMMs~\cite{Boyd}, \cite{Bertsekas}) to handle the
coupling constraints that arise from the dynamics of the system and inexact
solvers (i.e., solvers that guarantee feasibility and optimality only
\textit{asymptotically} with the number of iterations), such as the Nesterov's
Dual Fast Gradient (DFG) method~\cite{Nesterov1}. In particular, the first step
of the proposed algorithm is to split the original MPC problem over the length
$N$ of the prediction horizon into $N+1$ independent subproblems
(\textit{time-splitting}~\cite{StathopoulosECC13}) solved by $N+1$
parallel \emph{workers} periodically exchanging information at predetermined
synchronization points.
Then, the second step is to solve these subproblems in parallel using an inexact solver and guarantee, at the same time, that the solution of the original MPC problem is \emph{recursively feasible} and the system is \emph{closed-loop stable}.
The combination of parallelization and inexact solvers can result in
infeasibility and closed-loop instability.
We rely on an algorithm based on constraint tightening to overcome
these issues.
Loosely speaking, constraint-tightening algorithms solve an alternative
problem in which the constraints have been tightened by a certain amount to
compensate for the accuracy loss (and possible related infeasibility) introduced
by the solver. We rely on an \textit{adaptive} tightening strategy to select an
appropriate amount of tightening for our algorithm. Every time
new measurements are available from the plant, our algorithm chooses the amount
of tightening required for each subproblem in order to compensate for the error
introduced by the time-splitting combined with the inexact solver.

\textbf{Related work}. The \textit{time-splitting} technique has been
proposed in~\cite{StathopoulosECC13}. In contrast to~\cite{StathopoulosECC13},
we combine ADMM with inexact solvers and focus on the requirements for recursive
feasibility and closed-loop stability of the original problem.

Other constraint-tightening schemes have been proposed in the literature
(outside the parallel framework). For example, the authors in~\cite{RugPat:13} propose an algorithm in
which the amount of tightening is chosen offline to guarantee
suboptimality and feasibility of the solution for all the initial states of the MPC problem.
Solutions based on adaptive constraint tightening have been recently proposed in
\cite{NecoaraOCAM}, where the tightening parameter is chosen adaptively. Compared 
to~\cite{NecoaraOCAM}, our tightening update rule allows for a nonuniform
amount of tightening (the tightening varies along the prediction horizon).
Furthermore, thanks to the modular structure of our approach, the optimizer solves
simpler problems of fixed dimension, which is independent from $N$. As a
consequence, an increase of $N$ does not affect the conditioning of the problem
and the convergence of the solver. Hence, our approach leads to a
performance improvement even when forcing full serialization of the parallel
operations (i.e., \emph{serialized} mode~\cite{CS}).

\textit{\bf Outline}. {In the following,
Section~\ref{sec:problem_formulation} presents the initial problem formulation.
Section~\ref{sec:subproblem_reformulation} introduces the auxiliary
subproblems and our proposed solver. Section~\ref{sec:tightening_improvements}
describes our strategy to select the tightening of each subproblem to handle the parallelization error.
Section~\ref{sec:closed_loop_certification} proposes an online update strategy of the tightening parameters that 
guarantees recursive feasibility, suboptimality, and closed-loop stability.
Section~\ref{sec:evaluation} presents numerical results using an academic
example.
Finally, Section~\ref{sec:conclusion} concludes the paper.}
  
\textit{\bf Notation}.
For $u\in \mathbb R^n$, $\left \| u \right
\|=\sqrt{\langle u,u \rangle}$ is the Euclidean norm and $\left[ u \right]_+$ is
the projection onto non-negative orthant $\mathbb R^n_+$. Given a matrix $A$,
$[A]_i$ denotes the $i$-th row of A and $[A]_{i,j}$ the entry $(i,j)$ of A.
 Furthermore, $\mathbf 1_{n}$ is the vector of ones in $\mathbb R^n$ and 
 $I_n$ the identity matrix in $\mathbb R^{n\times n}$. In addition,
 $\textrm{eig}_{\max}(A)$ and $\textrm{eig}_{\min}(A)$ denote the largest and
 the smallest (modulus) eigenvalues of the matrix A, respectively. $P\in
 \posdef{n}$ denotes that $P\in \mathbb R^{n\times n}$ is positive definite.
\section{Problem Formulation}
\label{sec:problem_formulation}     
\hyphenation{sub-pro-blem}
Consider the discrete-time linear system described below:
\begin{equation}
\label{eq:LTI_system}
x(t+1) = Ax(t) +Bu(t)\quad\quad \forall t\geq 0,
\end{equation}   
where $x(t)\in \mathcal X \subseteq \mathbb R^n$ denotes the state of the
system and $u(t)\in \mathcal U \subseteq \mathbb R^m$ denotes the control
input. The sets $\mathcal X$ and $\mathcal U$ are simple proper convex sets
(i.e., convex sets that contain the origin in their interior).
Our goal is to steer $x(t)$ to the origin and satisfy
the plant constraints. We use MPC to achieve these
objectives.
In this respect, consider the following finite-time optimal control problem:
\begin{subequations} 
\label{eq:initial_problem}
\begin{align}
\label{eq:cost_fun}
\mathcal{V}^*(x_{\textrm{init}})=&{\underset{x,u}{\text{min}}}~\frac{1}{2}
\sum\limits_{t=0}^{N-1} (x_t^{T}Qx_t^{}\!+\!u_t^{T}Ru_t^{})+x_N^TP_Nx_N\\
\label{eq:dyn} \text{s.t.: }&
x_{t+1}^{} = A x_t^{}+B
 u_t^{},\quad~~\timeint{0}{N-1}\\
  \label{eq:set}
 &C x_t^{}+D u_t^{}+g\le 0,~\timeint{0}{N-1} \\
 & x_0=x_{\textrm{init}}\\
 \label{eq:terminal_set} 
  &x_N\in\mathcal{X}_N.
\end{align}
\end{subequations}
where $x_t$ and $u_t$ are more compact notations for $x(t)$ and $u(t)$,
respectively.
For $\timeint{0}{N-1}$ ($N$ denotes the prediction horizon), the states
and the control inputs are constrained in the polyhedral set described by~\eqref{eq:set}, where $C\in\mathbb{R}^{p_t\times n}$, $D\in\mathbb{R}^{p_t\times m}$, $g\in\mathbb{R}^{p_t}$. Note that \eqref{eq:set} can include constraints on the state only, i.e., $x_t\in \mathcal{X}$, and/or constraints on the control inputs only,
i.e., $u_t\in \mathcal{U}$. In~\eqref{eq:cost_fun}, $Q\in
\possemdef{n}$ and $R\in \posdef{m}$.
Our problem formulation considers also a terminal cost $x_N^TP_Nx_N$ associated with a terminal polyhedral set
 $\mathcal{X}_N~:=~\{x\in\mathbb R^n|F_N x\mathbf\le f_N, F_N\in \mathbb R^{p_N\times n}, f_n\in \mathbb R^{p_N}\}$.

Through the remaining of the paper, we assume:
\begin{assumption}
\label{ass:stabilizability}
 The pair $(A,B)$ is stabilizable.
 \end{assumption}
 \vspace{-0.3cm}
 \begin{assumption}
 \label{ass:terminalSet}
 Suppose Assumption~\ref{ass:stabilizability} holds. Given the gain $K_f\in
 \mathbb{R}^{m\times n}$ obtained by the infinite-horizon linear quadratic regulator (IH-LQR)---characterized by the matrices $A$, $B$, $Q$, and $R$---the following holds:
\begin{align*}
\forall x\in \mathcal{X}_N~\Rightarrow~&\begin{cases} x\in\mathcal{X},~ K_f u\in\mathcal{U},~\text{and}\\
                                       (A+BK_f)x\in \mu\mathcal{X}_N,~0\le\mu<1.\end{cases}  
\end{align*}
In addition, the terminal penalty $P_N\in \posdef{n}$ in the stage
cost~\eqref{eq:cost_fun} is defined by the solution of the algebraic
Riccati equation associated with the IH-LQR.
\end{assumption} 
In general, the MPC controller solves the optimization
problem~\eqref{eq:initial_problem} every time new measurements are available
from the plant and returns an optimal sequence of states and control inputs that
minimizes the cost~\eqref{eq:cost_fun}. Let the optimal sequence be defined as
follows:
\begin{equation}
\label{eq:consolidated_prediction}
\{\mathbf x,\mathbf u\} := \{x_0,\ldots,x^*_N, u_0^*,\ldots,u_{N-1}^*\}.
\end{equation}
Only the first element of
$\mathbf u$ is implemented in closed-loop, i.e., the control law obtained using
the MPC controller is given by:
\begin{equation}
\label{eq:MPC_control_law}
\kappa_{\text{MPC}}(x_{\text{init}}) = u_0^*, 
\end{equation}
and the closed-loop system is described by
\begin{equation}
x(t+1) = Ax(t)+B\kappa_{\text{MPC}}(x_{\text{init}}).
\end{equation} 
 \subsection{Parallelization}
\label{subsec:parallelization}
We aim to solve Problem~\eqref{eq:initial_problem} in parallel. Hence, we
exploit a similar approach as the one proposed in~\cite{StathopoulosECC13}. Specifically, as in~\cite{StathopoulosECC13}, 
Problem~\eqref{eq:initial_problem} is decomposed along the length of the
prediction horizon $N$ into $N+1$ independent subproblems to be solved by $N+1$
\emph{parallel workers} $\Pi_t$ (\timeint{0}{N}). Each $\Pi_t$ is
allowed to communicate with its neighbours $ \Pi_{t-1}$ and $\Pi_{t+1}$ at
predefined synchronization points.
The decomposition is possible thanks to the introduction of $N$ auxiliary
variables $z_t~(\timeint{1}{N})$ used to break the dynamic coupling that arises from~\eqref{eq:dyn}. These $z_t$ can be seen as the global variables of the algorithm. In particular, each $z_t$ stores the local predicted state $x_{t+1}$ of
 each subproblem and exchanges this stored information to guarantee consensus
 between neighboring subproblems, i.e., to ensure that the predicted state of
 the $(t)$-th subproblem, namely $x_{t+1}^{(t)}$, is equal to the current state
 of the $(t\!+\!1)$-st subproblem, namely $x_{t+1}^{(t+1)}$. Specifically, by
 introducing the {\it consensus constraints} $z_{t+1}\! =\! x_{t+1}^{(t)}\!
 =\! x_{t+1}^{(t+1)}$, defining $y_t := [x_{t}^{(t)T} ~ u_{t}^{(t)T}]^T$, $H_1:=
 [I_n~0]$, $H_2:= [A~B]$, $\rho>0$  Problem~\eqref{eq:initial_problem} becomes:
\begin{subequations}
\label{eq:consensus_problem}
\begin{align} 
\label{eq:cost_fun_consensus}
&{\underset{y,z}{\text{min}}}~ \sum\limits_{t=0}^{N} \mathbf{V}_t\big
(y_t, z_t\big)  \\
&\label{eq:set_consensus}\text{s.t.: }
 G_t y_t+g_t\le\! 0,\quad~ \timeint{0}{N-1}\\
 &~~~~~~ H_1y_0=x_{\textrm{init}},\\
 &~~~~~~ H_1y_N\in\mathcal{X}_{N}^{},\\ 
  \label{eq:consensus_1}
  &~~~~~~ z_{t+1} = H_2y_{t},\quad~\timeint{0}{N-1},\\
  \label{eq:consensus_2}
   &~~~~~~ z_{t+1} = H_1y_{t+1},~\timeint{0}{N-1},
\end{align}
\end{subequations} 
where, defining $\xi_t:=[y_t^T z_t^T z_{t+1}^T]^T$: 
\ifdraft
{\color{black}
\begin{itemize}
  \item $\mathbf{V}_0(\xi_0) : = \frac{1}{2}
  \xi_0^T\mathcal Q_0 \xi_0$, where 
  \begin{equation*}
  Q_0 := \begin{bmatrix}\mathcal H_0+\frac{\rho}{2} H_2^TH_2 &
  \frac{\rho}{2}H_2\\
-\frac{\rho}{2}H_2 & \rho I_n
  \end{bmatrix}
  \end{equation*}
  and $H_0:=\text{diag}\{Q,R\}$.
  \item $\mathbf{V}_t(\xi_t) : = \frac{1}{2}
  \xi_t^T\mathcal Q_t \xi_t$, where, for $t=1,\ldots,N-1$,
  \begin{equation*}
  Q_t := \begin{bmatrix}\mathcal H_0+\frac{\rho}{2} (H_2^TH_2+H_1^TH_1) &
  \frac{\rho}{2}H_1 & \frac{\rho}{2}H_2\\
-\frac{\rho}{2}H_1 & \rho I_n &0\\
-\frac{\rho}{2}H_2 & 0& \rho I_n
  \end{bmatrix}
  \end{equation*}
  \item $\mathbf{V}_N(\xi_N) :=\frac{1}{2}
  \xi_N^T\mathcal Q_N \xi_N$ where 
 \begin{equation*}
  Q_N := \begin{bmatrix}\mathcal H_N+\frac{\rho}{2} H_1^TH_1 &
  \frac{\rho}{2}H_1\\
-\frac{\rho}{2}H_1 & \rho I_n
  \end{bmatrix}
  \end{equation*}
 and $H_N := \text{diag}\{P_N,0_{m\times m}\}$.
\end{itemize}} 
\else
\begin{itemize}
  \item $\mathbf{V}_0(\xi_0) : = \frac{1}{2}
  \xi_0^T\mathcal Q_0 \xi_0 = \frac{1}{2}(y_0^TH_0 y_0+\rho\|H_2y_0-z_1\|^2)$,
  where $H_0:=\text{diag}\{Q,R\}$.
  \item $\mathbf{V}_t(\xi_t) : = \frac{1}{2}
  \xi_t^T\mathcal Q_t \xi_t = \frac{1}{2}(y_t^TH_0
 y_t+\rho\|H_2y_t-z_{t+1}\|^2+\rho\|H_1y_t-z_t\|^2)$, $t=1,\ldots,N-1$,.
  \item $\mathbf{V}_N(\xi_N) :=\frac{1}{2}
  \xi_N^T\mathcal Q_N \xi_N = \frac{1}{2}(y_N^TH_N
 y_N+\rho\|H_1y_N-z_N\|^2)$, where $H_N = \text{diag}\{P_N,0_{m\times m}\}$.
\end{itemize}
\fi
Furthermore, $G_t$ and $g_t$ vary for each subproblem as follows:
\begin{itemize}
  \item $G_t := [C~D]$ and $g_t:=g$, $t = 0,\ldots,N-1$.
  \item $G_N := [F_N~0_{p_N\times m}]$ and $g_N:= -f_N$.
\end{itemize}
\begin{rem}Note that we introduced a quadratic penalty in the cost of the form
$\rho/2(\|H_1y_t-z_t\|^2+\|H_2y_t-z_{t+1}\|^2)$,
according to the ADMM strategy~\cite{Boyd}. This penalty has no impact on the
cost of the original problem~\eqref{eq:initial_problem}, if the consensus
constraints are satisfied.
\end{rem}
In the following, we introduce the subproblems that derive from
Problem~\eqref{eq:consensus_problem}. Let $v_{t+1}$ ($\timeint{0}{N-1}$) and
$w_{t}$ ($\timeint{1}{N}$) be the Lagrange multipliers associated with the
equality constraints~\eqref{eq:consensus_1} and~\eqref{eq:consensus_2}, respectively.
Then, let the augmented
Lagrangian with respect to the multipliers $v_{t+1}$ and $w_t$ be defined as follows:
\begin{align*}
\mathcal{L}_{v_{t+1},w_t}\!\! := & \bfV_t(\xi_t)\!+\!\rho[
v_{t+1}^T(H_2y_t\!-\!z_{t+1})+ w_{t}^T(H_1y_t\!-\!z_{t})].
\end{align*}
Hence, we obtain the following $N+1$ independent subproblems, called
\textit{original} subproblems, associated with the $N+1$ workers $\Pi_t$
(\timeint{0}{N}):
\begin{align}
\label{eq:subproblem}
&{\underset{y_t,z_t}{\text{min}}} ~\mathcal{L}_{v_{t+1},w_t}(y_t,z_t,z_{t+1})
~\text{s.t.: } G_t y_t+ g_t^{}\le 0.
\end{align}
\subsection{Overview of our proposed approach and terminology}
\label{subsec:overview}
\begin{figure}
   \centering
    \includegraphics[width =1 \columnwidth]{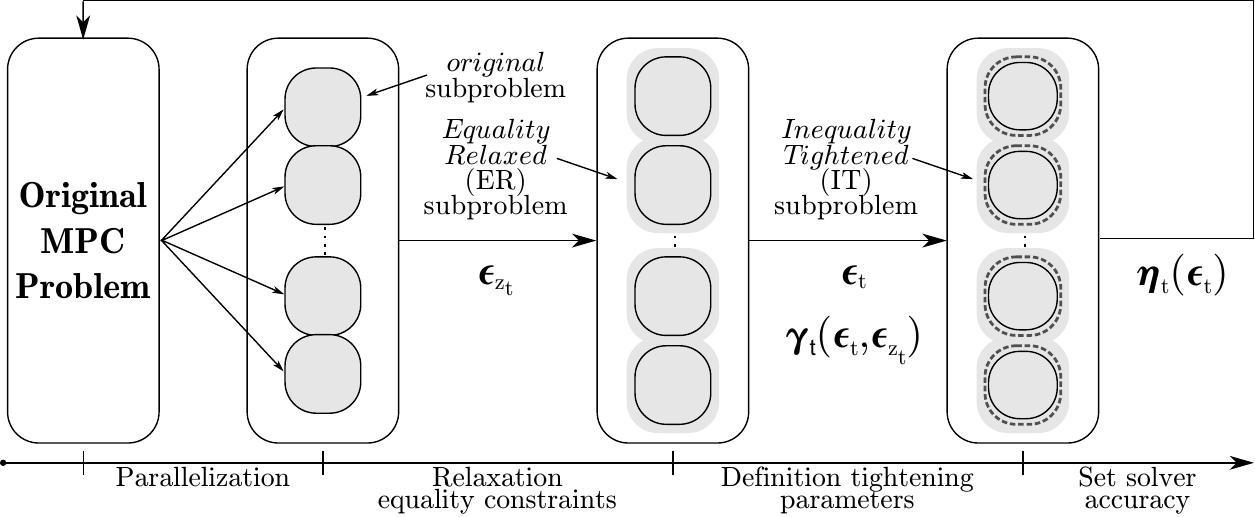}       
    \caption{Notation and terminology.} 
   \label{fig:steps}       
   \vspace{-0.3cm}
 \end{figure}
Figure~\ref{fig:steps} summarizes the main steps that lead to a suboptimal solution of the aforementioned problem and introduces the keywords used in 
the remaining of the paper.
Consider the MPC problem~\eqref{eq:initial_problem} and refer to this problem
as the \textit{original} MPC problem. The first step (\textit{parallelization},
detailed in Section~\ref{subsec:parallelization}) is to rewrite the original
problem in $N+1$ independent subproblems (the \textit{original} subproblems).
The aim is to use a Dual Fast Gradient (DFG) method to solve these subproblems
in order to certify---in terms of suboptimality and recursive primal
feasibility, along with closed-loop stability---the MPC solution.
The use of the inexact solver will eventually cause a violation of the consensus
constraints~\eqref{eq:consensus_1}-\eqref{eq:consensus_2} introduced to define
the subproblems~\eqref{eq:subproblem}. Hence, the
second step (\textit{relaxation equality constraints} detailed in Section~\ref{subsec:equality_relaxation}) is introduced to relax the consensus constraints by a quantity\footnote{Note that the subscript $t$ indicates that $\epsilon_{z_t}$ varies along the prediction horizon. This also holds for the later-defined $\epsilon_t$,
$\gamma_t$, and $\eta_t$.} $\epsilon_{z_t}$, preventing the occurrence of 
consensus-constraint violations. We refer to these subproblems as the
\textit{equality relaxed} (ER) subproblems. The set of inequality constraints of the ER subproblems includes 
the inequality constraints of the original subproblems~\eqref{eq:subproblem} and the
inequality constraints due to the relaxation of the consensus
constraints~\eqref{eq:consensus_1}-\eqref{eq:consensus_2}. The third step (\textit{definition tightening parameters} detailed in Sections~\ref{subsec:up_lagrange_multiplier},~\ref{sec:tightening_improvements}, and~\ref{sec:closed_loop_certification})
is required to address the following remaining issues.
First, the solution of each ER subproblem computed by the dual fast gradient method might violate the inequality constraints, due to 
the termination of the solver after a finite number of iterations. Hence, the constraints
of the ER subproblems must be tightened by a quantity $\epsilon_t$ proportional to the desired
level of suboptimality $\eta_t$ chosen by the algorithm. Second, due to the relaxation of the
consensus constraints, the \textit{consolidated}
prediction, i.e., the predicted evolution of
the state computed (a posteriori) using the control sequence obtained 
by the independent subproblems, might deviate from
the predicted \textit{local} solution computed by the independent
subproblems and, eventually, violate the inequality constraints of the original problem.
Hence, an additional tightening (dependent on $\epsilon_{z_t}$) must be
introduced on the subset of inequality constraints of the ER subproblems that corresponds to the original inequality constraints. 
The proposed algorithm addresses the aforementioned issues by exploiting the
\textit{inequality tightened} (IT) subproblems.
The IT subproblems differ from the ER subproblems in the definition of the feasibility region, which is tightened by a quantity 
$\gamma_t(\epsilon_t,\epsilon_{z_t})$ that depends on both $\epsilon_t$ and $\epsilon_{z_t}$. 
The last step (\textit{set solver accuracy}) selects a suboptimality level $\eta_t$ for each subproblem, 
that guarantees a primal feasible and suboptimal solution for the original MPC problem within a fixed 
number of iterations. 
\section{Subproblem reformulation}
\label{sec:subproblem_reformulation}
In the following, we introduce the ER subproblems and the proposed algorithm to
solve them using $N+1$ parallel workers. Furthermore, we introduce an initial
formulation of the IT subproblems and derive conditions on the choice of the
relaxation and tightening parameters to guarantee primal feasible and suboptimal solutions for of each
subproblem.  
\vspace{-0.2cm}
\subsection{Equality constraint relaxation}
\label{subsec:equality_relaxation}
Our goal is to obtain a solution for Problem~\eqref{eq:consensus_problem} by solving the 
independent subproblems~\eqref{eq:subproblem} in parallel using inexact solvers, such
as the Nesterov's DFG~\cite{Nesterov1}. In order to use our proposed solver
(introduced in Section~\ref{sec:alg1}), which relies on first-order methods, we
introduce a reformulation of Problem~\eqref{eq:consensus_problem} to take
into account that the constraints~\eqref{eq:consensus_1}
and~\eqref{eq:consensus_2} cannot be satisfied at the equality due to the
iterative nature of the proposed solver and its asymptotic
convergence properties.
In particular, introducing the
relaxation parameters $\epsilon_{z_{t}},\!\epsilon_{z_{t+1}}\!\!>\!\!0$, for
each subproblem ($\timeint{0}{N-1}$), the former equality
constraints~\eqref{eq:consensus_1}-\eqref{eq:consensus_2} are replaced by the
following inequality constraints:
\vspace{-0.2cm}
\begin{subequations}
\label{eq:relaxed_consensus_inequ}
\begin{align}
|H_1y_t- z_{t}| \leq \epsilon_{z_{t}}\1_n&\Leftrightarrow |x_{t}^{(t)}-z_t| \leq
\epsilon_{z_{t}}\1_n, \label{eq:relaxed_consensus_1}\\
|H_2y_t\!-\! z_{t+1}| \!\leq\! \epsilon_{z_{t+1}}\1_n&\Leftrightarrow
|x_{t+1}^{(t)}\!-\! z_{t+1}| \leq\epsilon_{z_{t+1}}\1_n. \label{eq:relaxed_consensus_2}
\end{align}
\end{subequations}
Thus, for each subproblem, we can realistically consider a feasible region defined by the following constraints:
\vspace{-0.2cm}
\begin{subequations}
\begin{align}
\label{eq:original_inequality_constraints}
&[~G_t~~~~0 ~~~~~~0]~\xi_t+~g_t~~~~~~\leq 0,\\
\label{eq:relaxed_equality_constraints_w_plus}
&[~H_1~-I_n~~~~0]~~\xi_t-\epsilon_{z_t}\mathbf{1}_{n}~~\leq 0,\\
\label{eq:relaxed_equality_constraints_w_minus}
&[-H_1 ~~~ I_n ~~~~ 0]~~\xi_t-\epsilon_{z_t}\mathbf{1}_{n}~~\leq 0,\\
\label{eq:relaxed_equality_constraints_v_plus}
&[~H_2 ~~~~ 0~-I_n]~~\xi_t-\epsilon_{z_{t+1}}\mathbf{1}_{{n}}\leq 0,\\
\label{eq:relaxed_equality_constraints_v_minus}
&[-H_2~~~0~~~~ I_n]~\xi_t-\epsilon_{z_{t+1}}\mathbf{1}_{{n}}~\leq 0,
\end{align}
\end{subequations}
or, in a more compact notation:
\begin{equation}
\label{eq:relaxed_inequality_constraints}
G_{\xi_t} \xi_t+g_{\xi_t} \leq 0,
\end{equation}
where $G_{\xi_t}\in \mathbb R^{p_{\xi_t}\times (n+m)+2n}$ and $p_{\xi_t}
:=p_t+4n$.

In the remaining of the paper, we consider the following {\it equality relaxed}
(ER) subproblems:
\begin{align}
\label{eq:relaxed_subproblem}
\bfV_t^*={\underset{\xi_t}{\text{min}}}~ \bfV_t(\xi_t)~\text{s.t.:}~ G_{\xi_t}
\xi_t+g_{\xi_t} \!\leq\! 0,~\timeint{0}{N}.
\vspace{-1cm}
\end{align}
Hence, let 
$\mu_t:=[\lambda_t^T~{w_t^+}^T~{w_t^-}^T~{v_{t+1}^+}^T~{v_{t+1}^-}^T]^T\in \mathbb R^{p_{\xi_t}}_+$ 
be the Lagrange multiplier associated with the new set of inequality constraints defined 
by~\eqref{eq:relaxed_inequality_constraints}, where $\lambda_t$, ${w_t^+}$ ${w_t^-}$, 
${v_{t+1}^+}$, and ${v_{t+1}^-}$ are the multipliers associated with the 
constraints~\eqref{eq:original_inequality_constraints},
~\eqref{eq:relaxed_equality_constraints_w_plus},
~\eqref{eq:relaxed_equality_constraints_w_minus},
~\eqref{eq:relaxed_equality_constraints_v_plus}, 
and~\eqref{eq:relaxed_equality_constraints_v_minus}, respectively. 
Then, to handle the complicated
constraints~\eqref{eq:relaxed_inequality_constraints}, define, for each subproblem, the dual 
function $d_t(\xi_t,\mu_t)$ as follows:
\begin{equation}
\label{eq:dual_fun}
d_t(\xi_t,\mu_t) = {\underset{\xi_t}{\text{min}}} ~\mathcal{L}_{\mu_t}(\xi_t),
\end{equation}
where $\mathcal{L}_{\mu_t}(\xi_t) := \bfV_t(\xi_t)+\mu^T_{t}\text{diag}\{I_{p_t},\rho I_{4n}\}(G_{\xi_t}\xi_t+g_{\xi_t})$. 
We refer to \eqref{eq:dual_fun} as the {\it inner subproblem}. Hence, we aim to
solve the following dual subproblems ({\it outer subproblem}) in parallel to obtain
a solution for Problem~\eqref{eq:consensus_problem}:
\begin{equation}
\label{eq:dual_probl}
d_t^*= {\underset{\mu_t}{\text{max}}} ~{d}_t(\xi_t,\mu_t),~\timeint{0}{N}.
\end{equation}
\begin{rem}
\label{rem:modularity}
The size of the ER subproblems remains unaffected if $N$ increases. Intuitively,
this {\it modularity} is an additional feature that can be exploited to
preserve some favorable numerical properties of the problem (e.g., conditioning,
Lipschitz constant, etc.) even when the algorithm is running in serialized mode.
\end{rem}
\subsection{Parallel dual fast gradient method}
\label{subsec:PDFG}
This section introduces Algorithm~\ref{alg:Alg1} that we use 
to solve the MPC problem~\eqref{eq:consensus_problem} exploiting $N+1$ parallel
workers $\Pi_t$. Note that, at this stage, we cannot yet ensure that the
computed solution is feasible and suboptimal for Problem~\eqref{eq:initial_problem}.
 
Algorithm~\ref{alg:Alg1} relies on Nesterov's DFG in
which the inner problem is solved in an ADMM fashion, as explained below. 
Specifically, \ifdraft {\color{black} as Figure~\ref{fig:alg1} depicts,}~\fi at each iteration of
the algorithm\footnote{Note that $[\xi_0^0]_{1:n}=x_0$ and
$\mu_t^0 =0_{p_{\xi_t}}$ to initialize Algorithm~\ref{alg:Alg1}.}, $\Pi_t$ computes a minimizer $\xi_t^k$ for
$\mathcal{L}_{\mu_t}(\xi_t)$ (steps 1-4), i.e., the algorithm returns a solution for each inner 
 subproblem~\eqref{eq:dual_fun}. In particular, our algorithm, 
 in compliance with the ADMM strategy, first minimizes $\mathcal L_{\mu_t}$ with
  respect to $y_t$ in parallel for each subproblem (step 1). Then, using the
  information received by $\Pi_{t+1}$, i.e., the updated value of
  $y^{k+1}_{t+1}$ (synchronization step 2), our algorithm computes---in
  parallel for each subproblem---the value of the global variable $z_t$ 
  according to the following rule (step 3):
\begin{equation}
\label{eq:update_z}
z_t^{k+1}\!=\!
\frac{1}{2}({H_1y_{t}^{k+1}\!+\!H_2y_{t-1}^{k+1}\!+\!v_t^+\!+\!w_t^+\!-\!v_t^-\!-\!w_t^-}).
\end{equation}
Note that this strategy allows to handle the coupling introduced by the 2-norm in the cost function of~\eqref{eq:relaxed_subproblem}. 
Then, (synchronization step 4) $\Pi_t$ receives (sends) the updated value of
$z_{t+1}^{k+1}$ ($z_t^{k+1}$) from $\Pi_{t+1}$ (to $\Pi_{t-1}$), respectively.
Finally, the worker~$\Pi_t$ computes the new values of the multipliers
$\mu_t^{k+1}$ (steps 5-7). We compute offline (for each subproblem) the
Lipschitz constant $L_{\mu_t}$ associated with $\nabla_{\mu_t}
d_t(\xi_t,\mu_t)$ to perform the multipliers' update:
\vspace{-0.1cm}
\begin{itemize}
\item $L_{\mu_0}=  \text{diag}
\bigg\{\frac{ \|G_0\|^2_2}{\textrm{eig}_{\textrm{min}}(\mathbf Q_0)}I_{p_0},\frac{ \|\rho\text{diag}\{ H_2,-I\}\|^2_2}{\textrm{eig}_{\textrm{min}}(\mathbf
Q_0)}I_{2n}\bigg\}$.
\item $ L_{\mu_t}=  \text{diag}
\bigg\{\frac{ \|G_t\|^2_2}{\textrm{eig}_{\textrm{min}}(\mathbf Q_t)}I_{p_t},\frac{ \|\rho\text{diag}\{ H_1, -I\}\|^2_2}{\textrm{eig}_{\textrm{min}}(\mathbf
Q_t)}I_{2n},\\
\frac{ \|\rho\text{diag}\{ H_2,-I\}\|^2_2}{\textrm{eig}_{\textrm{min}}(\mathbf
Q_t)}I_{2n}\bigg\}$, ($t=1,\ldots,N-1$).
\item $
L_{\mu_N}=  \text{diag}
\bigg\{\frac{ \|G_N\|^2_2}{\textrm{eig}_{\textrm{min}}(\mathbf Q_N)}I_{p_N},
\frac{ \|\rho\text{diag}\{ H_1, -I\}\|^2_2}{\textrm{eig}_{\textrm{min}}(\mathbf
Q_N)}I_{2n}\bigg\}$.
\end{itemize}
\captionsetup[algorithm]{font=small}
\begin{algorithm}[t]
\fontsize{8}{8}\selectfont
\begin{algorithmic} 
\State{Given $\xi_t^0, \mu_t^0,Q_t, G_t,g_t, L_{\mu_t}$, and $\bar k_t$ for each
$\Pi_t$~(\timeint{0}{N})}
\While{$k\le\bar k_t$}
 \State{{1.}~$\Pi_t$ computes $y_t^{k+1}\!
 =\!\operatorname{argmin}_{y_t}~\mathcal{L}_{\mu_t}(\xi_t^k,\mu_t^k).$}
\State{{2.}~$\Pi_t$ receives $y_{t+1}^{k+1}$ from $\Pi_{t+1}$, send
$y_{t}^{k+1}$ to $\Pi_{t-1}$.}
 \State{{3.}~$\Pi_t$ updates $z_t^{k+1}$ according to~\eqref{eq:update_z}.}
 \State{{4.}~$\Pi_t$ receives $z_{t+1}^{k+1}$ from $\Pi_{t+1}$, sends
 $z_t^{k+1}$ to $\Pi_{t-1}$.} \State{{5.}~$\Pi_t$ computes
\vspace{-0.3cm} 
 \[\hat \mu_t^{k+1}=\bigg[\mu^k_t+L_{\mu_t}^{-1}\nabla^T_{\mu_t}
 d_t(\xi_t^{k+1},\mu_t)\bigg]_{+}.\]
 \vspace{-0.3cm} }
  \State{{6.}~Define: $a:=\! \frac{k+1}{k+3}I$, $b_{\mu_t}\!:=\! \!L_{\mu_t}^{-1}\frac{2}{(k+3)}$.}
  \State{{7.}~$\Pi_t$ computes: 
  \vspace{-0.3cm} 
 \[\mu_t^{k+1}\!=\!a\hat
 \mu_t^{k+1}\!+\!b_{\mu_t}\bigg[\!\sum_{s=0}^{k}\frac{s+1}{2}\nabla^T_{\mu_t}
 d_t(\xi_t^{s},\mu_t)\bigg]_{+}.\] }
 \EndWhile
 \caption{Parallel Dual Fast Gradient Method.}
 \label{alg:Alg1}
 \end{algorithmic} 
\end{algorithm} 
\label{sec:alg1}   
\ifdraft
\begin{figure}[t]
   \centering
    \includegraphics[width=1\columnwidth]{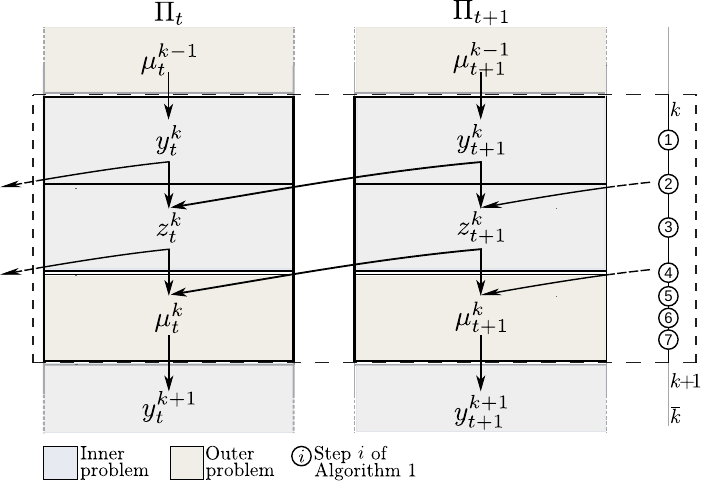}
    \caption{\color{black} Iteration $k$ of Algorithm~\ref{alg:Alg1}.}
   \label{fig:alg1}
 \end{figure}
\fi
Note that our update rule is different from the one proposed in~\cite{KogelCCA2011}, 
where ADMM is used in combination with Nesterov's fast gradient methods. At each
iteration, the algorithm proposed in~\cite{KogelCCA2011}, first computes the
exact minimizer $y_t$ and then updates $v_t$ and $w_t$. Our algorithm does not
wait until the DFG returns a minimizer $y_t$ to update the multipliers $v_t$ and
$w_t$, but starts updating their values along with the DFG iterations,
encouraging the information exchange between neighboring subproblems. This
algorithm is also different from the one proposed in~\cite{StathopoulosECC13}.
In particular, the workers exchange the necessary pieces of information
before the update of the Lagrange multipliers and none of the dual variables is
exchanged between the neighboring workers\ifdraft{\color{black}, as
Figure~\ref{fig:alg1} highlights}\fi.
Furthermore, the information exchange between neighboring workers is
unidirectional, i.e., $\Pi_{t+1}$ sends the updated information to $\Pi_t$, but
$\Pi_t$ does not send any updated information to $\Pi_{t+1}$.

Using an argument similar to the one of Theorem 1 in~\cite{NecoaraOCAM}, we can compute the primal feasibility 
violation and the level of suboptimality of the solution of each ER subproblem
returned by Algorithm~\ref{alg:Alg1}.
\vspace{-0.2cm}
\begin{thm}
\label{thm:thm1}
(\cite{NecoaraOCAM}) Let $\mathbf{V}_t(\xi_t)$ be strongly convex, the sequences
($\boldsymbol \xi_t ^k$, $\hat \mu_t^k$, $\mu_t^k$) be generated by Algorithm~\ref{alg:Alg1}, and
$\boldsymbol {\hat \xi}_t^k :=
\sum_{s=0}^k\frac{2(s+1)}{(k+1)(k+2)}\boldsymbol \xi^s$. Then, an estimate on
the primal feasibility violation for the ER subproblem~\eqref{eq:relaxed_subproblem} is given by the following:
\begin{equation}
\label{eq:primal_feasibility_violation}
\|[\nabla_{\mu_t}^T d_t(\boldsymbol{\hat\xi}_t^{k})]_+\|\le
\frac{8R_{t}{\max\{L_{\mu_t}\}}}{(k+1)^2}=:\eta_t,
\end{equation}
where $R_t:=\|\mu^*_t\|$ .
Moreover, an estimate on primal suboptimality is given by the following:
\begin{equation}
\label{eq:primal_suboptimality}
0\leq \mathbf V_t^* - \mathbf V_t(\boldsymbol{\hat
\xi}_t) \leq  R_{t}\eta_t.
\end{equation}
\end{thm}
Algorithm~\ref{alg:Alg1} terminates after a fixed number of iterations that depends on
$\eta_t$ and $R_t$~\cite{NecoaraOCAM}:
\begin{equation}\label{eq:iter_number} \bar k_t :=\bigg\lfloor
\sqrt{{8R_{t}\eta_t^{-1}{\max\{L_{\mu_t}\}}}}\bigg\rfloor.\end{equation}
\subsection{Tightening of the original inequality constraints}
\label{subsec:up_lagrange_multiplier}  
In order to guarantee the primal feasibility of each subproblem using
Algorithm~\ref{alg:Alg1}, we introduce $N+1$ auxiliary subproblems,
namely the {\it inequality tightened} (IT) subproblems, which differ from the
ER subproblems~\eqref{eq:relaxed_subproblem} in the definition of the feasible
region. In particular, each IT subproblem can be defined as follows:
\begin{equation}
\bfV_{\epsilon_t}^*\!=\!\min_{\xi_{t}}
\bfV_t(\xi_{t})~\text{s.t.:}~ G_{\xi_t}
\xi_{t}+g_{\xi_t}+\mathbf{\epsilon}_{t}\mathbf{1}_{p_t+4n} \leq 0,
\label{eq:tightened_subproblem}
\end{equation}
where $\epsilon_t\ge 0$ is the tightening parameter, which depends on the
suboptimality level~$\eta_t$ that the proposed algorithm
can reach within $\bar k_t$ iterations~\eqref{eq:iter_number}. According
to~\cite{NecoaraOCAM}, solving~\eqref{eq:tightened_subproblem} using
Algorithm~\ref{alg:Alg1} ensures, with a proper choice of $\epsilon_t$, 
that the solution of~\eqref{eq:tightened_subproblem} is primal feasible and suboptimal for
subproblem~\eqref{eq:relaxed_subproblem}.

To define an $\epsilon_t$ similar to the one introduced in~\cite{NecoaraOCAM},
we must compute an upper bound for the optimal Lagrange multiplier, namely
$\mu^*_{t,\epsilon_t}$, associated with the IT subproblem~\eqref{eq:tightened_subproblem}. We use an argument similar to the one of Lemma 1 in~\cite{Nedic}. 
In particular, we compute the aforementioned upper bound for
$\mu^*_{t,\epsilon_t}$ according to the following lemma.
\begin{lem}
\label{lem:lemma_up_bound}
Assume that there exists a Slater vector $\tilde y_t\in\mathbb{R}^{n+m}$ such
that $G_t\tilde y_t+g_t< 0$. Then, there exists $\epsilon_t\geq 0$,
$\epsilon_t<{\min}_{j=1,\ldots,p_t}\{-(G_t\tilde y_t
+g_t)_j\}$, $\epsilon_{z_t},\epsilon_{z_{t+1}}> \epsilon_t$, such that the
upper bound for $\mu^*_{t,\epsilon_t}$ is given by
\begin{equation}
\label{eq:Rd_definition}
\|\mu_{t,\epsilon_t}^*\|\leq 2{R_{d_t}}:= 2{\frac{\mathbf V_t(\tilde
\xi_t)-d_t(\tilde
\mu_t)}{\min\limits_{j=1,\ldots,p_t+2n}\{[\Gamma_t]_j\}}},
\end{equation}
where $\Gamma_t\!\! := \!\!\!\big[[-(G_t\tilde y_t
\!+\!g_t)^T\!-\!\epsilon_t\1_{p_t}^T] [2\rho(\epsilon_{z_t}-\epsilon_t)\1_{n}^T]
[2\rho(\epsilon_{z_{t+1}}-\epsilon_t)\1_{n}^T]\big]^T\in \mathbb R^{p_t+2n}$,
and $d_t(\tilde\mu_t)$ is the dual function for the original
subproblem~\eqref{eq:dual_probl} evaluated at $\tilde \mu_t\in \mathbb
R^{p_t+4n}$.
\end{lem}
\begin{proofs}
\ifdraft
See Appendix~\ref{app:lemma1}.
\else
See Appendix I in~\cite{TechReport}.
\fi
\end{proofs}
\begin{rem} Lemma~\ref{lem:lemma_up_bound} does not only provide an upper bound
for $\|\mu_{t,\epsilon_t}^*\|$, but it also provides guidelines to select the
values of $\epsilon_{z_{t}}$ and $\epsilon_{z_{t+1}}$ as a function of 
  ${\min}_{j=1,\ldots,p_t}\{-(G_{t}\tilde y_t+g_{t})_j\}$, which only depends on
  the primal variable $\tilde y_t$. An alternative way to determine the relaxation parameters is to
  include $\epsilon_{z_{t}}$ and $\epsilon_{z_{t+1}}$  in the set of decision variables and 
  penalize them in the cost function as it is usually done to handle soft
  constraints. This will, however, increase the number of decision
  variables in the problem formulation and it will have an impact on the original cost.
\end{rem}   
\section{Tightening improvement to guarantee primal feasible
consolidated predictions}
\label{sec:tightening_improvements}
\ifdraft
\begin{figure} 
   \centering
    \includegraphics[width=1\columnwidth]{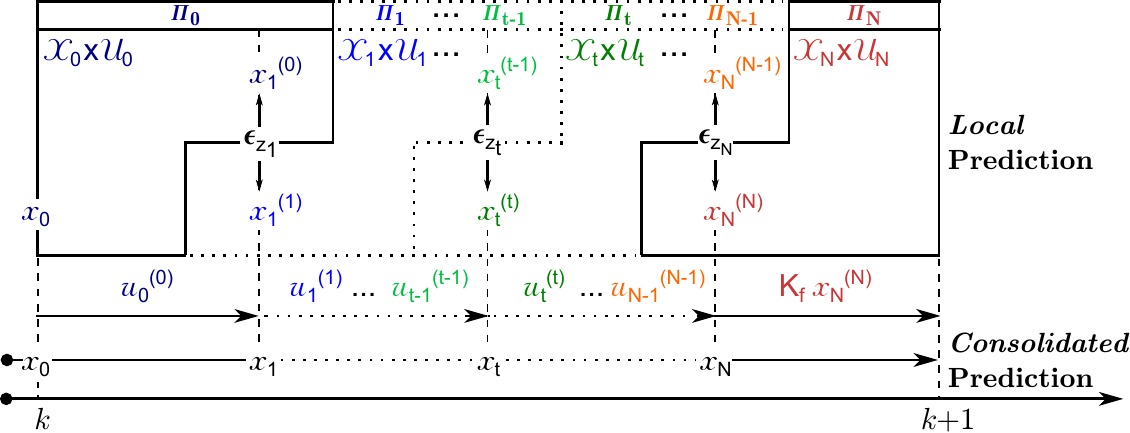}
    \caption{\color{black}\textit{Local} and \itglob~predictions. The different colors
    highlight the different subproblems (e.g., the dark blue color refers
    to the subproblem handled by worker~$\Pi_0$, the blue color refers to
    subproblem handled by worker~$\Pi_1$, etc.).}
   \label{fig:local_global_predictions}
 \end{figure}
 \fi 
The previous section showed how to choose the tightening parameter $\epsilon_t$
of each IT subproblem to ensure that the $t$-th \textit{local} solution,
i.e., the solution computed by the $t$-th IT
subproblem~\eqref{eq:tightened_subproblem}, is primal feasible for the $t$-th ER subproblem. This section provides guidelines to improve the choice
of the tightening parameter of each IT
subproblem~\eqref{eq:tightened_subproblem} in order to guarantee the primal
feasibility of the \itglob solution, i.e., the predictions obtained, starting from the initial state $x_0$, using the control sequence 
\vspace{-0.2cm}
\begin{equation}
\label{eq:tightened_control}
 \bar\ubf_{\epsilon}:=
\{\bar u_{0,{\epsilon_0}}^{(0)},\ldots,\bar
u_{N-1,\epsilon_{N-1}}^{(N-1)}\},\end{equation} where the elements of
$\bar\ubf_{\epsilon}$ are computed by the independent IT
subproblems~\eqref{eq:tightened_subproblem}.\ifdraft{\color{black}~Figure~\ref{fig:local_global_predictions} highlights the difference between the local and the consolidated prediction.}\fi ~In particular, when a new measurement is available from the plant, the subsystems~\eqref{eq:tightened_subproblem} compute in parallel $(x_{0},
\bar u_{0,\epsilon_0}^{(0)}),$ $\ldots$,$(
\bar x_{N-1,\epsilon_{N-1}}^{(N-1)},\bar u_{N-1,\epsilon_{N-1}}^{(N-1)}),$ and
$x_{N,\epsilon_N}^{(N)}$, respectively.
According to the results of the previous section, the pair
$(\bar x_{t,\epsilon_t}^{(t)}, \bar u_{t,\epsilon_{t}}^{(t)})$ is primal
feasible for the $t$-th subproblem~\eqref{eq:relaxed_subproblem}, thanks to the introduction of the IT subproblems. Nevertheless, due to the relaxation introduced on the equality constraints~\eqref{eq:relaxed_consensus_1}-\eqref{eq:relaxed_consensus_2}, there is a bounded mismatch between $x_{t+1}^{(t)}$ and $x_{t+1}^{(t+1)}$ ($\timeint{0}{N-1}$). Hence, starting from the initial state $x_0$, when the control sequence $\bar \ubf_{\epsilon}$ is applied
to compute the consolidated state prediction~
\vspace{-0.2cm}
\begin{equation}
\label{eq:tightened_state}
\mathbf {\bar
x}_{\epsilon}:=\{x_{0},\bar x_{1,\epsilon_1}\ldots,\bar x_{N,\epsilon_N}\},
\end{equation}
the feasibility of $\mathbf {\bar x}_{\epsilon}$ is no longer guaranteed.
Note, however, that $\bar\ubf_{\epsilon}\in\mcU:=\mcU_1\times\ldots\times
\mcU_{N}$, i.e., $\bar\ubf_{\epsilon}$ is feasible. Hence, no additional
tightening is needed on the input constraints.
 
In the following, Section~\ref{subsec:upbound} defines an upper bound on the
maximal feasibility violation of $\mathbf {\bar x}_{\epsilon}$. This
feasibility violation is a consequence of the local relaxations of the equality constraints.
Then, Section~\ref{subsec:improve_tightening} introduces sufficient
conditions to ensure the primal feasibility of the \itglob prediction
and provides guidelines for the choice of the tightening parameters for each IT
subproblem.
\subsection{Upper bound on the maximal feasibility violation of $\mathbf
{\bar x}_{\epsilon}$}
\label{subsec:upbound}
Let $\bar\ubf_\epsilon$ and $\mathbf {\bar x}_\epsilon$ be defined
by~\eqref{eq:tightened_control} and~\eqref{eq:tightened_state}, respectively.
Moreover, from~\eqref{eq:relaxed_consensus_1}
and~\eqref{eq:relaxed_consensus_2}, the following holds:
\begin{equation}
\label{eq:local_mismatches}
|\bar x_{t,\epsilon_{t-1}}^{(t-1)}-\bar x_{t,\epsilon_t}^{(t)}|\le 2
\epsilon_{z_t}.
\end{equation}
Our goal is to characterize how far the \itglob predicted state is
from the \itloc predicted state.
\begin{lem}
\label{lem:upper_bound_global_prediction}
Let the $t$-step-ahead \itglob prediction $\bar x_{t,\epsilon_t}$ be defined
by~\eqref{eq:tightened_state} and assume that~\eqref{eq:local_mismatches}
holds. Then, there exists $\alpha_t\in\mathbb R, \alpha_t\geq 0$, such that the
mismatch between $\bar x_{t,\epsilon_t}$ and the state of the $t$-th subproblem
$\bar x_{t,\epsilon_t}^{(t)}$ is bounded, as follows:
\begin{equation}
\label{eq:up_bound_alpha}  
|\bar x_{t,\epsilon_t}-\bar x_{t,\epsilon_t}^{(t)}| \le \alpha_t. 
\end{equation}     
\end{lem}
\begin{proofs}
\ifdraft
See Appendix~\ref{app:lemma2}.
\else
See Appendix~II in~\cite{TechReport}.
\fi  
\end{proofs}
\begin{rem}
\label{rem:after_lemma_2_a}
According to Lemma~\ref{lem:upper_bound_global_prediction} a possible choice of
$\alpha_t$ is the following:
\vspace{-0.3cm}
\begin{equation}
\label{eq:alpha_t}
\alpha_t:=2\sum_{j=0}^{t-1}\|A^j\|\epsilon_{z_{t-j}}.
\end{equation}
\end{rem}
\subsection{Tightening parameter selection}
\label{subsec:improve_tightening}
According to Lemma~\ref{lem:upper_bound_global_prediction},
$\bar x_{t,\epsilon_t}$ differs from $\bar x_{t,\epsilon_t}^{(t)}$ by a quantity
bounded by $\alpha_t$. Thus, $\bar x_{t,\epsilon_t}$ might violate the
constraints of the $t$-th subproblem~\eqref{eq:subproblem} by as much as $\alpha_t$, in the worst-case scenario. In particular, we must ensure that
$C_t\bar x_{t,\epsilon_t}+D_t \bar u_{t,\epsilon_t} +g_t\le 0$. Using the
computed upper bound~\eqref{eq:up_bound_alpha}, the following holds: 
\begin{align*}
\begin{array}{c}
C_t \bar x_{t,\epsilon_t}^{(t)}+D_t\bar u_{t,\epsilon_t}^{(t)} +g_t +|C_t|
\alpha_t\1_n +\epsilon_t\1_{p_t}\le 0 \\
~~~~~~\Updownarrow{\text{\eqref{eq:up_bound_alpha}}}\\
C_t \bar x_{t,\epsilon_t}+D_t \bar u_{t,\epsilon_t}^{(t)} +g_t +|C_t|
\alpha_t\1_n +\epsilon_t\1_{p_t}\le 0,
\end{array} 
\end{align*} 
where $|C_t|$ indicates the absolute value of $C_t$. 
Recall that these mismatches are caused by the use of inexact
solvers and that $\alpha_t$ depends on $\epsilon_{z_t}$. In the
following, we provide guidelines to improve the choice of $\epsilon_t$ for each
subproblem. Furthermore, we provide a modified upper bound for the optimal
Lagrange multiplier associated with the tightened subproblems~\eqref{eq:tightened_subproblem}, which considers the additional tightening introduced by $\alpha_t$.
\begin{lem}
\label{lem:improve_epsilon_t}
Consider the following IT subproblems:
\begin{equation}
\label{eq:tightened_subproblem_gamma}
\bfV^*_{\gamma_t}\!=\!\min_{\xi_{t}}~
\bfV_{t}(\xi_{t})~\text{s.t.:}~ G_{\xi_{t}}
\xi_{t}+g_{\xi_t}+\mathbf{\gamma}_{t} \leq 0,
\end{equation}
for $\timeint{0}{N}$, where
$\gamma_t$$:=[$$(|C_t|\alpha_t\1_n+\epsilon_t\1_{p_t})^T$~$\epsilon_t\1^T_{4n}]^T$.
Consider the assumptions of Lemma~\ref{lem:lemma_up_bound} and the existence of
$\alpha_t$ for all $\timeint{1}{N}$ according to
Lemma~\ref{lem:upper_bound_global_prediction}.
Then, for each subproblem, there exist $\epsilon_t\geq 0$,
$\epsilon_{z_t}, \epsilon_{z_{t+1}}>\epsilon_t$ such
that the upper bound for the optimal Lagrange multiplier associated with the IT
subproblems~\eqref{eq:tightened_subproblem_gamma} is described by
\begin{equation*}
\|\mu_{t,\gamma_t}^*\|\leq 2\mathcal{R}_t := 2{\frac{\mathbf V_t(\tilde
\xi_t)-d_t(\tilde
\mu_t)}{\min_{j=1,\ldots,p_t+2n}\{[\Gamma_{\alpha_t}]_j\}}},~\timeint{0}{N},
\end{equation*}
$
\Gamma_{\alpha_t}:=\big [
[-(G_{t}\tilde
y_t+g_{t})^T-(|C_t|\alpha_t\1_n)^T-\epsilon_t\mathbf{1}_{p_t}^T]
[2\rho (\epsilon_{z_{t}}-\epsilon_t)\mathbf{1}_n^T]
[2\rho(\epsilon_{z_{t+1}}-\epsilon_t)\mathbf{1}_n^T]\big]^T\in
\mathbb R^{p_t+2n}$.
\end{lem}
\begin{proofs}
\ifdraft
See Appendix~\ref{app:lemma3}.
\else
See Appendix~III in~\cite{TechReport}.
\fi
\end{proofs}
\begin{rem}
\label{rem:after_lemma_3}
The choice of $\epsilon_t$ ($\timeint{0}{N}$) is not unique and depends on the
choice of $\epsilon_{z_t}$ ($\timeint{1}{N}$).
For example, given $\alpha_t$ in~\eqref{eq:alpha_t}, a
possible choice of  $\epsilon_{z_t}$ ($\timeint{1}{N}$) is:
\begin{align}
\epsilon_{z_t}\!\leq\!\!
\min&\bigg\{\!\frac{\epsilon_{z_N}}{\|A^{N-t}\|},\!..,\!\frac{\epsilon_{z_{t+1}}}{\|A\|},\!\frac{\upbound{t}}{1+2t\!\!\!\maxC{t}}\!\bigg\}.
\label{eq:choice_epsilon_z}
\end{align}
Consequently, the tightening parameters are given by:
\begin{align}
\epsilon_t \!\leq\!
\frac{1}{2}\min\Bigg\{\!\epsilon_{z_t},\epsilon_{z_{t+1}},\!\!{\upbound{t}}\!\Bigg\},
\label{eq:choice_epsilon_t}
\end{align}
for $\timeint{0}{N}$. This choice implies that first we select the relaxation
parameters and then we \textit{adapt} the tightening parameters on the
original inequality constraints based on the choice of $\epsilon_{z_t}$ for all
$\timeint{1}{N}$. An alternative is to fix $\epsilon_t$ for the inequality
constraints and consequently compute $\epsilon_{z_t}$. In general, the choice
of the parameters strongly depends on the system-state matrix $A$
in~\eqref{eq:LTI_system}.
\end{rem}
\begin{rem}
\label{rem:offline_computation_parameters}
In the context of this work, Algorithm~\ref{alg:Alg2}, described in the next
section, adapts the above derived parameters at each problem instance. If we
consider a fixed tightening scheme, such as the one proposed by~\cite{RugPat:13}, $\epsilon_{t}$
and $\epsilon_{z_t}$ can be computed offline (for all the
initial states in the region of attraction).
\end{rem}
In the following, we show that by using $\{\mathbf{\bar
x}_\gamma, \bar\ubf_{\gamma}\}$---$\bar\ubf_\gamma$ is the control sequence
obtained by solving the IT subproblems~\eqref{eq:tightened_subproblem_gamma} and
$\mathbf{\bar x}_\gamma$ is the corresponding consolidated prediction---the
inequality constraints of the original MPC problem~\eqref{eq:initial_problem}
are satisfied. Consequently, the predicted final state is in the terminal set of
the original problem.
If the desired level of suboptimality of Algorithm~\ref{alg:Alg1} is chosen as:
\begin{equation}
\label{eq:accuracy_t}
\eta_t :={\epsilon_t}/{2},
\end{equation}
then, according to Theorem~\ref{thm:thm1}, there exists $
\bar\xi_{t,{\gamma_t}}:=[\bar y_{t,{\gamma_t}}^{T}~ \bar
z_{{t},{\gamma_t}}^T~ \bar z_{{t+1},{\gamma_t}}^T]^T$ such that
$\|[\nabla^T_{\mu_t}d_{\gamma_t}(\bar \xi_{t,{\gamma_t}})]_+\|
\leq\eta_t<\epsilon_t$.
Using similar arguments as in~\cite{NecoaraOCAM}, the following holds for $\timeint{0}{N}$:
\begin{align*}
\Bigg[G_{\xi_t}\bar \xi_{t,{\gamma_t}}+g_{\xi_t}+\begin{bmatrix}
|C_t|\alpha_t\1_n+\epsilon_t\1_{p_t}\\
\epsilon_t\1_{4n}
\end{bmatrix}\Bigg]_+<\epsilon_t\1_{p_t+4n}.
\end{align*}
Hence, for all $j=1,\ldots,p_t$, the following holds
\begin{align*}
\bigg[[C_t \bar x_{t,{\gamma_t}}^{(t)}+D_t
\bar
u_{t,{\gamma_t}}^{(t)}+g_t+|C_t|\alpha_t\1_n+\epsilon_t\1_{p_t}]_j\bigg]_+\leq \epsilon_t.
\end{align*}
Consequently, exploiting the upper bound~\eqref{eq:up_bound_alpha},
for all $j=1,\ldots,p_t$, we have:
\begin{align*}
&[C_t \bar x_{t,{\gamma_t}}+D_t
\bar u_{t,{\gamma_t}}+g_t+|C_t|\alpha_t\1_n+\epsilon_t\1_{p_t}]_j\leq \epsilon_t
\end{align*}
which leads to $C_t \bar x_{t,{\gamma_t}}+D_t
\bar u_{t,{\gamma_t}}+g_t<0\quad\forall \timeint{0}{N}$, where
$\bar x_{t,{\gamma_t}}$ is the $t$-step-ahead \itglob prediction computed using
the solution to the IT subproblem~\eqref{eq:tightened_subproblem_gamma} with
tighening parameter $\gamma_t$.

In summary, this section showed that there exists a choice of the relaxation 
and tightening parameters that guarantee a feasible consolidated prediction with respect to the
original problem~\eqref{eq:initial_problem}.
\section{Suboptimality, recursive feasibility, and closed-loop
stability guarantees}
\label{sec:closed_loop_certification}

In the following, we derive bounds for $\bfV_{\gamma}:=\sum_{t=0}^N \bfV_t(\mathbf{\bar x}_{\gamma},\bar
{\mathbf u}_{\gamma})$, i.e., the cost obtained using $\{\mathbf{\bar
x}_\gamma, \bar\ubf_{\gamma}\}$, with respect to the optimal cost $\mathcal{V}^*$ of the original problem. 
\begin{thm}
\label{th:theorem2}
Assuming that there exist $\epsilon_t$ ($\timeint{0}{N}$) and $\epsilon_{z_t}$
($\timeint{1}{N}$) selected according to Lemma~\ref{lem:improve_epsilon_t}, then the following holds:
\vspace{-0.5cm}
\begin{equation}
\label{eq:up_low_cost_fun}
\mathcal V^*\le\bfV_{\gamma}\le\mathcal
 V^*+2\sum_{t=0}^N\mathcal R_t\sqrt{p_t}\bar
\gamma_t,
\end{equation}
where $\bar \gamma_t:=\epsilon_t+\maxC{t}\alpha_t$.
\end{thm}
\begin{proofs}
\ifdraft
See Appendix~\ref{app:proof_theorem_2}.
\else
See Appendix IV in~\cite{TechReport}.
\fi
\end{proofs}
Theorem~\ref{th:theorem2} established the level of suboptimality of the consolidated prediction with respect to the original problem.
In particular, the sequence $\{\bar {\mathbf x}_{\gamma},\bar
{\mathbf u}_{\gamma}\}$ is suboptimal for the original problem and satisfies
the original inequality constraints (including those associated with $\mathcal
X_N$).
    
Recall that for the update of $\mathcal R_t$, our algorithm requires a strictly
feasible vector $\tilde y_t$ for~\eqref{eq:set_consensus}. Hence, every time new
measurements are available from the plant, our algorithm must provide a strictly
feasible solution (not necessarily optimal) for the first $p_t$ inequality
constraints of each ER subproblem. The following
lemma provides guidelines to compute $\tilde y_t$.
\begin{lem}
\label{lem:update_y}
Let $\bar y_{\gamma}$ be defined as  $\bar y_{\gamma}:=[\bar
y_{0,\gamma_0}^T\ldots\bar y_{N,\gamma_N}^T]=[(x_{0}^T~\bar
u_{0,\gamma_0}^T)\ldots(\bar x_{N-1,\gamma_{N-1}}^T~\bar u_{N-1,\gamma_{N-1}}^T)~(\bar
x_{N,\gamma_N}^T))]^T$. Then, a feasible $\tilde y^+$ at the next problem
instance, is given by:
\begin{align}
\label{eq:strictly_feasible_y}
\tilde y^+ & = [\bar y_{{\gamma}_{[2:N+1]}}((A+BK_f)\bar x_{N,\gamma_N})^T]^T 
\end{align}  
\end{lem}
\begin{proofs}
\ifdraft 
See Appendix~\ref{app:proof_lemma_update_y}.
\else
See Appendix V in~\cite{TechReport}.
\fi
\end{proofs}
We want to show that the cost decreases at each problem instance.
Using a similar argument as in~\cite{RawlingsBook}, under Assumption~\ref{ass:terminalSet} 
on $\mathcal X_N$ and ensuring that $\bar x_{N,\gamma_N}\in\mathcal X_N$ (thanks to a proper choice of the tightening
parameters, as the previous section showed), we can show that:
\begin{align}
\label{eq:upper_bound_feasible_solution_next_problem_instance}
\!\!\sum\limits_{t=0}^{N}\!\bfV_{t}(\tilde y_{t}^+)\! \leq\!
\sum\limits_{t=0}^{N}\! \bfV_{t}(\bar y_{t,\gamma_t})\!-\!\!\bfV_{0}(y_{0,\gamma_0})~\forall y_{0,\gamma_0}\!\in \!\mathcal Y_{\textrm{attr}},
\end{align}  
where $\mathcal Y_{\textrm{attr}}$ is the region of attraction. Hence,
from \eqref{eq:up_low_cost_fun}
and \eqref{eq:upper_bound_feasible_solution_next_problem_instance}, the following
holds:
\ifdraft
{\color{black}
\begin{subequations}
\label{eq:asymptotic_stability_inequality}
\begin{align}
&\sum\limits_{t=0}^N\bfV_{t}(\bar y_{t,\gamma_t}^+) \overset{\text{\eqref{eq:up_low_cost_fun}}}{\leq} \mathcal V^*(x^+)+  \sum\limits_{t=0}^{N}f(\bar\gamma_t^+,\mathcal R_t^+)\\
&\leq  \sum\limits_{t=0}^N\bfV_{t}(\tilde y_{t}^+)+  \sum\limits_{t=0}^{N}f(\bar\gamma_t^+,\mathcal R_t^+)\\
& \overset{\text{\eqref{eq:upper_bound_feasible_solution_next_problem_instance}}}{\leq}  
 \sum\limits_{t=0}^N\bfV_{t}(\bar y_{t,\gamma_t})-\bfV_{0}(y_{0,\gamma_0})+ \sum\limits_{t=0}^{N}f(\bar\gamma_t^+,\mathcal R_t^+)
\end{align}
\end{subequations}}
\else
\begin{align}
\label{eq:asymptotic_stability_inequality}
\!\!\!\!&\sum\limits_{t=0}^N\!\!\bfV_{t}(\bar y_{t,\gamma_t}^+){\leq} \!\!
 \sum\limits_{t=0}^N\!\!\big[\bfV_{t}(\bar
 y_{t,\gamma_t}\!)\!+\!f(\bar\gamma_t^+,\mathcal R_t^+)\big]\!\!-\!\!\bfV_{0}(y_{0,\gamma_0})
\end{align}
\fi
where $f(\epsilon_t^+,\mathcal
R_t^+,\alpha_t^+)\!\!:=\!\!(2\mathcal
R_t^+\sqrt{p_t})\bar \gamma_t^+$, using $\bar\gamma_t^+,\mathcal
R_t^+$ to represent the updated values of these parameters according to $\tilde
y^+_\gamma$. The inequality above shows that the total cost decreases at
each problem instance if $\mathcal X_N$ is defined according to
Assumption~\ref{ass:terminalSet} and if the $N$-step-ahead \itglob prediction
lies in the terminal set. Asymptotic stability of our controller follows if $\bfV_{0}(y_{0,\gamma_0})\geq \sum_{t=0}^{N}f(\bar\gamma_t^+,\mathcal
R_t^+)$.
Hence, we can modify the update of $\epsilon_t$ and $\epsilon_{z_t}$ to ensure
that~\eqref{eq:asymptotic_stability_inequality} is satisfied.
\begin{rem}
\label{rem:update_z}
A possible choice of $\epsilon_{z_t}$ $(\timeint{1}{N})$ to
fulfill~\eqref{eq:asymptotic_stability_inequality} is the following:
\begin{align}
&\epsilon_{z_t}^+ \leq
\min\bigg\{\bar \epsilon_{z_t},\epsilon_{z_t}~\text{in
\eqref{eq:choice_epsilon_z}}\bigg\},
\label{eq:choice_epsilon_z_update}
\end{align}
\begin{align*}
\bar \epsilon_{z_t}={\bfV_{0}}\bigg[{4N\mathcal R_t^+
\sqrt{p_t}\bigg(1+2t\maxC{t}\bigg)}\bigg]^{-1}.\notag
\end{align*}
Consequently, $\epsilon_t$ can be selected according to~\eqref{eq:choice_epsilon_t} to
preserve the definition of the upper bound on the optimal Lagrange multipliers
given in Lemma~\ref{lem:improve_epsilon_t}.
\end{rem}
Algorithm~\ref{alg:Alg2} summarizes the main steps needed to obtain a
stabilizing control law when the original MPC problem is solved in parallel
using inexact solvers. In particular, note that, if the measured
state is in $\mathcal X_N$, from Assumption~\ref{ass:terminalSet}, the state
and the control constraints are automatically satisfied without solving the MPC
problem in parallel.
\begin{algorithm}[t]
\fontsize{8}{8}\selectfont
\begin{algorithmic}[1]
\State{Given $A,B,\mathcal X,\mathcal U, \mathcal X_N, N$}
\State{Compute offline: $K_f, P_f, F_N, f_N$.}
\State{Measure: initial state $x_{\textrm{init}}$ at time $t=0$.}
\For{$t=0$ {\bf to} $N$}
\State{Compute offline: $G_{\xi_t}, g_{\xi_t}, \mathcal Q_t, \mathcal W_t, c_t$.}
\State{Compute: initial strictly feasible vector $\tilde y_t$.}
\State{Compute: initial tightening according to Lemma~\ref{lem:improve_epsilon_t}.}
\EndFor
\For{$t=0$ \textbf{to} $\infty$}
\State{Measure: initial state $x_{\textrm{init}}$.}
\If{$x_{\textrm{init}}\in\mathcal X_N$}
\State{Compute: $u=K_fx_{\textrm{init}}$.}
\Else
\State{Compute in parallel (Alg. 1):
$\bar\xi_{0,{\gamma_t}},\ldots,\bar\xi_{N,{\gamma_N}}$
exploiting~\eqref{eq:tightened_subproblem_gamma}.}
\State{Compute: $u=\bar u_{\gamma_0}$.}
\State{Update: $\tilde y\leftarrow \tilde y^+$ according to~\eqref{eq:strictly_feasible_y}.}
\For{$t=0$ \textbf{to} $N-1$}
\State{Update: $\epsilon_{z_{N\!-\!t}}\!\!\!\leftarrow\!
\epsilon_{z_{N\!-\!t}}^+$ according to Lemma~\ref{lem:improve_epsilon_t}.}
\EndFor
\For{$t=1$ \textbf{to} $N$}
\State{Update: $\epsilon_{t} \leftarrow \epsilon_{t}^+$
according to Lemma~\ref{lem:improve_epsilon_t}.}
\State{Update: $\gamma_{t} \leftarrow \gamma_{t}^+$
according to Lemma~\ref{lem:improve_epsilon_t}.}
\EndFor
\EndIf
\State{Implement $u$.}
\EndFor
 \caption{MPC with adaptive parallel tightening scheme.}
 \label{alg:Alg2}
 \end{algorithmic} 
\end{algorithm}
\begin{rem}
Steps 17-23 are the only nonparallel ones of the algorithm
(Algorithm~\ref{alg:Alg1} is instead fully parallelizable).
The main reason is in the adaptive nature of the algorithm (see also
Remark~\ref{rem:offline_computation_parameters}). Algorithm~\ref{alg:Alg2}
adapts $\epsilon_t$ and $\epsilon_{z_t}$ every time new measurements are available from the plant.
A fully parallel Algorithm~\ref{alg:Alg2} is possible using a fixed tightening strategy, in which $\epsilon_t$ and $\epsilon_{z_t}$ can be computed
offline.
\end{rem}
\section{Evaluation}
\label{sec:evaluation} 
We evaluated Algorithm~\ref{alg:Alg2} using the LTI system described
in~\cite{Example}. The system (sampled at $T_s = 0.5$ s) is
described by:
\begin{align*}
x(t+1) = A x(t) +B u(t),~h(t) = Cx(t)+D u(t),
\end{align*}
where $x(t)\!\in\! \mathcal X\!\!:=\!\!\big\{\!x(t)\!\in\! \mathbb R^2\big |
|x_i(t)|\!\leq\! 4 (i=1,2),\! \forall t\!\geq\! 0\big\}$, $u(t)\in \mathcal U
:=\big\{u(t)\in \mathbb R^2\big | |u_i(t)|\leq 1 ~(i=1,2), \forall t\geq 0\big\}$, $h(t)\in \mathcal
H:=\big\{h(t)\in \mathbb R^2\big| |h_i(t)|\leq 1~ (i=1,2), \forall t\geq
0\big\}$,
 and the quadruple $(A,B,C,D)$ is given by:
\begin{align*}
& A = \begin{bmatrix} 
1.09&0.22\\
0.49&0.02
\end{bmatrix},~B = \begin{bmatrix}
1.22&0.88\\
-0.78&-0.34
\end{bmatrix}\\
& C= \begin{bmatrix}
1.34&-0.16\\
-3.19&-0.56
\end{bmatrix},~D = \begin{bmatrix}
1.60&1.01\\
-0.68&0.77
\end{bmatrix}
\end{align*}
The weighting matrices $Q$,
$R$, and $P_N$ in the cost~\eqref{eq:cost_fun} and the IH-LQR gain
$K_f$ are selected according to~\cite{Example}. 
We implemented our design in MATLAB (to tune the controller and test the
initial design) and in C (to run a performance analysis).
In particular, in MATLAB, we used the Parallel Computing
Toolbox\texttrademark~to assign the computation of Algorithm~\ref{alg:Alg1} to 8 parallel workers, given a
prediction horizon $N=7$.
Furthermore, we relied on the MPT3 toolbox~\cite{MPT3} to compute
$\mathcal X_N$ and the optimal solution of
Problem~\eqref{eq:initial_problem}.
Finally, we compared our design to~\cite{NecoaraOCAM}.

We considered the following scenario. The initial state of the system is $x_0
= [-0.101 -3.7]^T$. The total number of complicated
constraints~\eqref{eq:set} for the original problem is 90.
We used~\eqref{eq:choice_epsilon_z}
and~\eqref{eq:choice_epsilon_t} to initialize
$\epsilon_{z_t}$ and $\epsilon_{t}$, respectively.
To update them, we relied on~\eqref{eq:choice_epsilon_z_update}
and~\eqref{eq:choice_epsilon_t}.
The selected $x_0$ caused $u$ and $h$ to saturate
(12 active constraints). In this scenario, the state enters $\mathcal X_N$ in 3 steps.

Figure~\ref{fig:upp_bound_alpha} shows the mismatch between the \itloc
prediction and the \itglob prediction for one problem instance. As
Figure~\ref{fig:upp_bound_alpha} depicts, the mismatch (for both states) is
below the predicted upper bound $\alpha_t$ for all the $N+1$ subproblems.
\begin{figure}[t]
   \centering
\includegraphics[scale=.6]{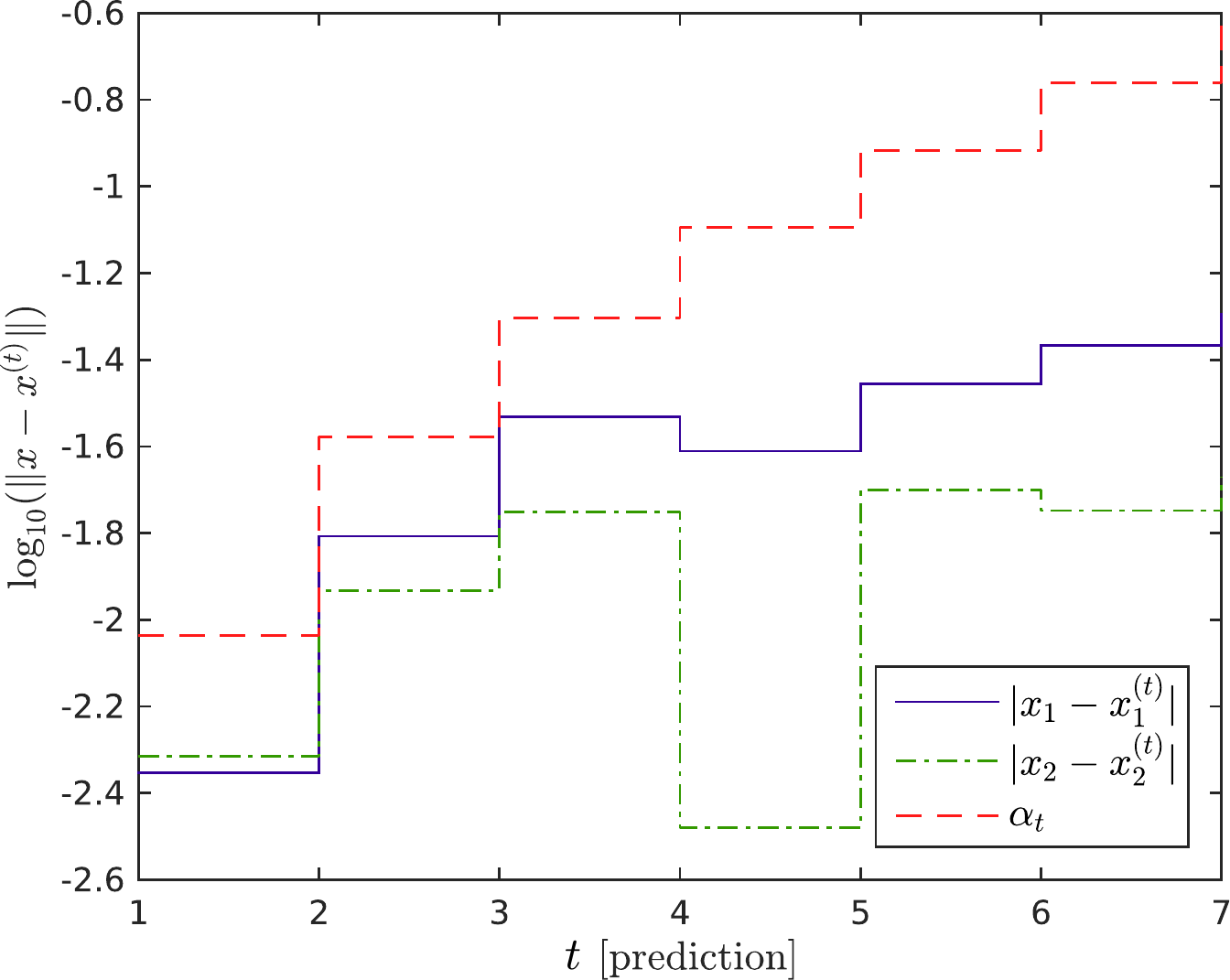}
\caption{Mismatch between \itloc and \itglob predictions.}
    \label{fig:upp_bound_alpha} 
    \vspace{-.5cm}  
 \end{figure} 
 
Table~\ref{tab:iterations} compares the proposed technique
to~\cite{NecoaraOCAM}. The table reports the upper bound $\bar k$ on the number
of iterations needed to achieve a suboptimal solution for
Problem~\eqref{eq:initial_problem} and the level of suboptimality
$\eta$. The table lists only the
  first four subproblems, which are the most significant due to the presence of
  active constraints in these subproblems.
  The method in~\cite{NecoaraOCAM} and our new proposed parallel algorithm
  produce a comparable behavior, thanks to an appropriate selection of the tightening parameters. 
The parallel algorithm, however, is able to achieve similar results to
those in~\cite{NecoaraOCAM} using a smaller
number of iterations.
The larger values of $\bar k$
for~\cite{NecoaraOCAM} are probably caused by the value of the Lipschitz
constant and by the problem conditioning, which affect the convergence requiring
a higher accuracy for the solver.
In particular, in our proposed framework, the DFG is applied to simpler problems
characterized, in the worst case scenario, by a Lipschitz constant $\max_{\timeint{0}{N}} \{L_{\mu_t}\} = 196$ and by a condition number $\max_{\timeint{0}{N}} 
\{\kappa_t\}= 104$. In~\cite{NecoaraOCAM}, the DFG
solves a larger problem characterized by a Lipschitz
constant $L_s = 21994$ and by a condition number $\kappa_s = 7020$. Hence, the
modularity of our approach has positive implications on important properties for
the convergence of the solver.

Table~\ref{tab:iterations} lists the time required by the optimizer to
return a suboptimal solution for Problem~\eqref{eq:initial_problem}.
To measure the performance, we implemented both algorithms in C
on a Linux-based OS. We noticed that given the small size of the problem, 
running Algorithm~\ref{alg:Alg1} in parallel did not result in significant
speedups compared to our algorithm running in serialized mode~\cite{CS}.
Nevertheless, in both cases, we registered a speedup (230x) compared
to~\cite{NecoaraOCAM}.
The modularity of the proposed algorithm is beneficial even for problems of
small size, such as the one considered in this section for the comparison
with~\cite{NecoaraOCAM}.
 We expect the benefits to be even more pronounced when considering problems of
 larger size.
 \captionsetup[table]{labelformat=simple, labelsep=newline,justification=centering, font=footnotesize}
\newcommand{\tablename}{TABLE}
\begin{table}[t]
\centering
  \caption{PERFORMANCE ANALYSIS OF ALGORITHM~\ref{alg:Alg1}
  AND~\cite{NecoaraOCAM}. RESULTS SHOW THE MEDIAN OF 11 EXPERIMENTS.}
\fontsize{7.5}{7.5}\selectfont
\renewcommand{\arraystretch}{1.4}
\centering
 \begin{tabular}[\columnwidth]{c|p{2.09cm}| p{1.72cm}  p{.05cm}  p{.05cm}
 p{.05cm} }
\hline
\multirow{1}{*}{ \textit{Sample}} & \multicolumn{5}{|c}{\it Iterations (for
subproblem)}\\
\multirow{1}{*}{ \textit{time}} & \multicolumn{1}{c}{$\bar k~{(\text{\cite{NecoaraOCAM}})}$}&
\multicolumn{1}{c}{$\bar k_0$} &\multicolumn{1}{c}{$\bar k_1$} & \multicolumn{1}{c}{$\bar k_2$}& 
 \multicolumn{1}{c}{$\bar k_3~~~\ldots$} \\  
 \hline
 0& $11\cdot 10^4$ (2.85 ms) & 58931 (1.85 ms)& 25   & 10 & 8
 \\
 1& $12\cdot 10^6$ (302.84 ms) & 18218 (0.57 ms) & 3480 & 0  & 0  \\
 2& $11\cdot 10^6$ (267.56 ms) & 2265 ( 0.07 ms) & 0    & 0  & 0  \\
 \cline{2-6}
 & \multicolumn{5}{|c}{\it Suboptimality Level (for subproblem)}\\
 & \multicolumn{1}{c}{$\eta~{(\text{\cite{NecoaraOCAM}})}$}& \multicolumn{1}{c}{$\eta_0$}
 &\multicolumn{1}{c}{$\eta_1$} & \multicolumn{1}{c}{$\eta_2$}& 
 \multicolumn{1}{c}{$\eta_3~~~\ldots$}\\
\cline{2-6}
  0 & \centering 3.48 &\centering  1.15 & \centering 1.15 &\centering  1.90
  &{\centering 2.25}\\
  1&  \centering 0.31 &\centering   0.51 & \centering 1.03 &\centering    1.21
  &{\centering  1.44}\\
 2& \centering 0.15 &\centering  0.62& \centering 1.25 &\centering  1.47 &
 {\centering 1.74 }\\
\hline
  \end{tabular}
  \label{tab:iterations}
  \vspace{-0.5cm}
\end{table} 
\section{Conclusions}   
\label{sec:conclusion}  
We proposed an algorithm tailored to MPC that guarantees recursive feasibility
and closed-loop stability, when the solution of the MPC problem is computed
using inexact solvers in a parallel framework. In particular, our algorithm combines
ADMM and DFG methods and relies on an adaptive
constraint-tightening strategy to certify the MPC law.

Our numerical analysis shows performance improvements compared to
 state-of-the-art nonparallel techniques~\cite{NecoaraOCAM}. Furthermore, our
 study shows that, for small-size problems, even if the solver is implemented in
 a serialized mode, there is substantial performance improvement with respect to
 the state of the art.
We expect further benefits from the parallelization when the size of the problem
increases.
A scalability analysis of the proposed algorithm on many-core architectures is
part of our ongoing work.

\ifdraft
\appendices
\setcounter{lem}{0}
\setcounter{thm}{1}
\section{Proof of Lemma~\ref{lem:lemma_up_bound}}   
\label{app:lemma1}
This section contains the proof of Lemma~\ref{lem:lemma_up_bound} presented in Section~\ref{subsec:up_lagrange_multiplier}.
\begin{lem}
Assume that there exists a Slater vector $\tilde y_t\in\mathbb{R}^{n+m}$ such
that $G_t\tilde y_t+g_t< 0$. Then, there exists $\epsilon_t,\epsilon_{z_t},
\epsilon_{z_{t+1}}\geq 0$,
$\epsilon_t<{\min}_{j=1,\ldots,p_t}\{-(G_t\tilde y_t
+g_t)_j\}$, $\epsilon_{z_t},\epsilon_{z_{t+1}}> \epsilon_t$, such that the
upper bound on $\mu^*_{t,\epsilon_t}$ is given by
\begin{equation*}
\label{eq:Rd_definition_appendix}
\|\mu_{t,\epsilon_t}^*\|\leq 2{R_{d_t}}:= 2{\frac{\mathbf V_t(\tilde
\xi_t)-d_t(\tilde
\mu_t)}{\min\limits_{j=1,\ldots,p_t+2n}\{[\Gamma_t]_j\}}},
\end{equation*}
where  
\begin{equation*}
\Gamma_t := \begin{bmatrix}
-(G_t\tilde y_t +g_t)-\epsilon_t\1_{p_t}\\
2\rho(\epsilon_{z_t}-\epsilon_t)\1_{n}\\
2\rho(\epsilon_{z_{t+1}}-\epsilon_t)\1_{n}
\end{bmatrix}\in \mathbb R^{p_t+2n}
\end{equation*}
and $d_t(\tilde\mu_t)$ is the dual function for the original
subproblem~\eqref{eq:dual_probl} evaluated at $\tilde \mu_t\in \mathbb
R^{p_t+4n}$.

\begin{proofs}
The following inequality holds:
\begin{align}
d(\tilde \mu_t)& \leq \bfV_t(\tilde \xi_t) +\mu_{t,\epsilon_t}^{*T}\nabla_{\mu_t}^T d_{t,\epsilon_t}(\mu_t) \notag\\
               & =  \bfV_t(\tilde \xi_t) + \lambda_{t,\epsilon_t}^{*T} (G_t\tilde y_t +g_t + \epsilon_t \mathbf{1}_{p_t})+\notag\\
               &~~~+ \rho w_{t,\epsilon_t}^{-*T}(-H_1\tilde y_t+z_t-\epsilon_{z_{t}}\mathbf{1}_n+ \epsilon_t \mathbf{1}_n)+\notag\\
               &~~~+ \rho  w_{t,\epsilon_t}^{+*T}(H_1\tilde y_t-z_t-\epsilon_{z_{t}}\mathbf{1}_n+ \epsilon_t \mathbf{1}_n)+\notag\\
               &~~~+  \rho v_{t+1,\epsilon_t}^{-*T}(-H_2\tilde y_t+z_{t+1}-\epsilon_{z_{t+1}}\mathbf{1}_n+ \epsilon_t \mathbf{1}_n)+\notag\\
               &~~~+  \rho v_{t+1,\epsilon_t}^{+*T}(H_2\tilde y_t-z_{t+1}-\epsilon_{z_{t+1}}\mathbf{1}_n+ \epsilon_t \mathbf{1}_n)\notag\\
               \label{eq:inequ_upper_bound}
               &\le  \bfV_t(\tilde \xi_t) + \lambda_{t,\epsilon_t}^{*T} (G_t\tilde y_t +g_t + \epsilon_t \mathbf{1}_{p_t})+\\
               &~~~+2 \rho~ \text{max}\{w_{t,\epsilon_t}^{-*T},w_{t,\epsilon_t}^{+*T}\}(-\epsilon_{z_{t}}\mathbf{1}_n+ \epsilon_t \mathbf{1}_n)+\notag\\
               &~~~+ 2 \rho~ \text{max}\{v_{t+1,\epsilon_t}^{-*T},v_{t+1,\epsilon_t}^{+*T}\}(-\epsilon_{z_{t+1}}\mathbf{1}_n+ \epsilon_t \mathbf{1}_n),\notag
\end{align}
where the last inequality takes into account that $w_{t,\epsilon_t}^{-*T},w_{t,\epsilon_t}^{+*T},v_{t+1,\epsilon_t}^{-*T},v_{t+1,\epsilon_t}^{+*T}\!\!\in\!\! \mathbb R^n_+$. Define $w_{t,\epsilon_t}^{*T}:=\max\{w_{t,\epsilon_t}^{-*T},w_{t,\epsilon_t}^{+*T}\}$ and $v_{t+1,\epsilon_t}^{*T}:=\max\{v_{t+1,\epsilon_t}^{-*T},v_{t+1,\epsilon_t}^{+*T}\}$. Consequently, using the definition of $w_{t}^{*T}$ and $v_{t+1}^{*T}$, the following holds:
\begin{equation}
\label{eq:lemma_mu}
\|\mu_{t,\epsilon_t}^*\|\le\|[\lambda_{t,\epsilon_t}^{*T}~w_{t,\epsilon_t}^{*T}~v_{t+1,\epsilon_t}^{*T}]^T\|.
\end{equation}
Furthermore, recalling that $\lambda_{t,\epsilon_t}^*\in \mathbb R^{p_t}_+$, $w_{t,\epsilon_t}^{*}\in \mathbb R^n_+$, and $v_{t+1,\epsilon_t}^{*}\in\mathbb R^n_+$, the following holds:
\begin{equation}
\label{eq:lemma_mu_1}
\|[\lambda_{t,\epsilon_t}^{*T}~w_{t,\epsilon_t}^{*T}~v_{t+1,\epsilon_t}^{*T}]^T\|\!\leq\![\lambda_{t,\epsilon_t}^{*T}~w_{t,\epsilon_t}^{*T}~v_{t+1,\epsilon_t}^{*T}]^T\!\!\1_{p_t+2n}.
\end{equation}
 Hence, if we compute an upper bound for the vector $[\lambda_{t,\epsilon_t}^{*T}~w_{t,\epsilon_t}^{*T}~v_{t+1,\epsilon_t}^{*T}]^T$, we obtain an upper bound for $\|\mu_{t,\epsilon_t}^*\|$. Thus, from the inequality~\eqref{eq:inequ_upper_bound}, it follows that:
\begin{align}
\label{eq:ineq_lemma_1}
&\begin{bmatrix} \lambda_{t,\epsilon_t}^{*}\\ w_{t,\epsilon_t}^{*}\\v_{t+1,\epsilon_t}^{*}\end{bmatrix}^T
\underbrace{\begin{bmatrix}
-(G_{t}\tilde y_t+g_{t})-\epsilon_t\mathbf{1}_{p_t}\\
2\rho (\epsilon_{z_{t}}-\epsilon_t)\mathbf{1}_n\\
2\rho(\epsilon_{z_{t+1}}-\epsilon_t)\mathbf{1}_n
\end{bmatrix}}_{\Gamma_t} \leq  \bfV_t(\tilde \xi_t)-d(\tilde \mu_t).
\end{align}
Notice that choosing 
$\epsilon_t<\underset{j=1,\ldots,p_t}{\min}\{-(G_t\tilde y_t +g_t)_j\}$,
$\epsilon_{z_t},\epsilon_{z_{t+1}}>\epsilon_t$, i.e., according to the
assumptions of the lemma, the elements of $\Gamma_t$ are all greater than zero.

Thus, using~\eqref{eq:lemma_mu} and~\eqref{eq:lemma_mu_1}, it follows:
\begin{align}
&\frac{1}{2}\underset{j=1,\ldots,p_t+2n}{\min}\{[\Gamma_t]_j\}\|\mu_{t,\epsilon_t}^*\|
\le\notag \\
\label{eq:lemma_1_proof} &\le \begin{bmatrix} \lambda_{t,\epsilon_t}^{*}& 
w_{t,\epsilon_t}^{*}& v_{t+1,\epsilon_t}^{*}\end{bmatrix} \Gamma_t.
\end{align}
 Consequently, the upper bound on the optimal Lagrange multiplier is given by:
\begin{align*}
\|\mu_{t,\epsilon_t}^*\| \le 2 \frac{\bfV_t(\tilde \xi_t)-d(\tilde
\mu_t)}{\underset{j=1,\ldots,p_t+2n}{\min}\{[\Gamma_t]_j\}}.
\end{align*}
\end{proofs}
\end{lem}
\section{Proof of Lemma~\ref{lem:upper_bound_global_prediction}}
\label{app:lemma2}
This section contains the proof of Lemma~\ref{lem:upper_bound_global_prediction} presented in Section~\ref{subsec:upbound}.
\begin{lem}
Let the $t$-step-ahead \itglob prediction $\bar x_{t,\epsilon_t}$ be defined
by~\eqref{eq:tightened_state} and assume that~\eqref{eq:local_mismatches}
holds. Then, there exists $\alpha_t\in\mathbb R, \alpha_t\geq 0$, such that the
mismatch between $\bar x_{t,\epsilon_t}$ and the state of $t$-th subproblem
$\bar x_{t,\epsilon_t}^{(t)}$ is bounded, as follows:
\begin{equation*}  
|\bar x_{t,\epsilon_t}-\bar x_{t,\epsilon_t}^{(t)}| \le \alpha_t. 
\end{equation*}     
\end{lem}
\begin{proofs} 
In the following, we omit the dependence from $\epsilon_t$ to simplify the
notation. The proof is constructive.
For $t=0$, $x_0 \equiv \bar x_0^{(0)}$.
For $t=1$, $\bar x_1 = A x_0+B\bar u_0\equiv \bar x_1^{(0)}$, which is the
1-step-ahead state computed by the local subproblem 0, i.e., the subproblem associated to worker $\Pi_0$. Hence, the mismatch
between $\bar x_1$ and $\bar x_1^{(1)}$ is simply given by
\begin{equation*}
|\bar x_1-\bar x_1^{(1)}|\leq 2\epsilon_{z_1}=\alpha_1.
\end{equation*} 
For $\timeint{2}{N}$, the following holds:
\begin{align*}
|\bar x_2 - \bar x_2^{(2)}| &= |\bar x_2-\bar x_2^{(1)}+\bar
x_2^{(1)}-\bar x_2^{(2)}|
\\
                  & \le |\bar x_2-\bar x_2^{(1)}|+|\bar
                  x_2^{(1)}-\bar x_2^{(2)}|\\
                  & \le|A \bar x_1^{(0)}+B
                  \bar u_1-A\bar x_1^{(1)}
                  -B\bar u_1|+2\epsilon_{z_2}\\
                  & \le
                  2(\|A\|\epsilon_{z_1}+\epsilon_{z_2})=\alpha_2,\\
                  &~\vdots
                  \end{align*}
                  \begin{align*}
&|\bar x_N -\bar x_N^{(N)}|  \le
2(\|A^{N-1}\|\epsilon_{z_1}+\|A^{N-2}\|\epsilon_{z_2}\\
& \quad\quad\quad\quad\quad\quad+\ldots+\epsilon_{z_N})=\alpha_N,
\end{align*}
which proves the lemma.  
\end{proofs}
\section{Proof of Lemma~\ref{lem:improve_epsilon_t}}
\label{app:lemma3}
This section contains the proof of Lemma~\ref{lem:improve_epsilon_t} presented in Section~\ref{sec:tightening_improvements}.
\begin{lem}
Consider the following IT subproblems:
\begin{equation}
\label{eq:tightened_subproblem_gamma_appendix}
\bfV^*_{\gamma_t}\!=\!\min_{\xi_{t}}~
\bfV_{t}(\xi_{t})~\text{s.t.:}~ G_{\xi_{t}}
\xi_{t}+g_{\xi_t}+\mathbf{\gamma}_{t} \leq 0,
\end{equation}
for $\timeint{0}{N}$, where
$\gamma_t$$:=[$$(|C_t|\alpha_t\1_n+\epsilon_t\1_{p_t})^T$~$\epsilon_t\1^T_{4n}]^T$.
Given the assumptions of Lemma~\ref{lem:lemma_up_bound} and the existence of
$\alpha_t$ for all $\timeint{1}{N}$ according to
Lemma~\ref{lem:upper_bound_global_prediction}.
Then, for each subproblem, there exist $\epsilon_t\geq 0$,
$\epsilon_{z_t}, \epsilon_{z_{t+1}}>\epsilon_t$ such
that the upper bound on the optimal Lagrange multiplier associated with the IT
subproblems~\eqref{eq:tightened_subproblem_gamma_appendix} is described by
\begin{equation*}
\|\mu_{t,\gamma_t}^*\|\leq 2\mathcal{R}_t := 2{\frac{\mathbf V_t(\tilde
\xi_t)-d_t(\tilde
\mu_t)}{\min_{j=1,\ldots,p_t+2n}\{[\Gamma_{\alpha_t}]_j\}}},~\timeint{0}{N},
\end{equation*}
\begin{equation*}
\Gamma_{\alpha_t}:=\begin{bmatrix}
-(G_{t}\tilde y_t+g_{t})-|C_t|\alpha_t\1_n-\epsilon_t\mathbf{1}_{p_t}\\
2\rho (\epsilon_{z_{t}}-\epsilon_t)\mathbf{1}_n\\
2\rho(\epsilon_{z_{t+1}}-\epsilon_t)\mathbf{1}_n
\end{bmatrix}.
\end{equation*}
\end{lem}
\begin{proofs}
This lemma follows from Lemma~\ref{lem:lemma_up_bound} applied to the
subproblems~\eqref{eq:tightened_subproblem_gamma}. From
inequality~\eqref{eq:ineq_lemma_1} formulated for
subproblem~\eqref{eq:tightened_subproblem_gamma}, the following must hold
\begin{align*}
&\begin{bmatrix} \lambda_{t,\epsilon_t}^{*}\\ w_{t,\epsilon_t}^{*}\\v_{t+1,\epsilon_t}^{*}\end{bmatrix}^T
\underbrace{\begin{bmatrix}
-(G_{t}\tilde y_t+g_{t})-|C_t|\alpha_t\1_n-\epsilon_t\mathbf{1}_{p_t}\\
2\rho (\epsilon_{z_{t}}-\epsilon_t)\mathbf{1}_n\\
2\rho(\epsilon_{z_{t+1}}-\epsilon_t)\mathbf{1}_n
\end{bmatrix}}_{\Gamma_{\alpha_t}} \leq \\
&\le \bfV_t(\tilde \xi_t)-d(\tilde \mu_t).
\end{align*}
Hence, in order to satisfy the inequality above, we can select the relaxation parameters $\epsilon_{z_t}$
and the tightening parameters $\epsilon_t$ according to the assumption of the
lemma for $\timeint{0}{N}$, i.e., the following must hold:
\begin{align*}
\text{\it (i)~}&\frac{1}{2}\!\!\min\limits_{j=1,\ldots,p_t}\!\!\{-(G_t\tilde
y_t+g_t)_j\}\!\!\geq\!\!\!\maxC{t}\alpha_t+\epsilon_t\\
\text{\it (ii)~}& \maxC{t}\alpha_t+\epsilon_t > 0\\
\text{\it (iii)~}&\epsilon_{z_t},\epsilon_{z_{t+1}}\geq \epsilon_t\geq 0,
\end{align*}
where $\tilde y_t$ is a strictly feasible solution for the original $t$-th
subproblem.
Hence, there exists $\mathcal R_t$ such that the upper bound on the optimal
Lagrange multiplier $\mu_{t,\gamma_t}^*$ is defined as follows:
\begin{align*}
&\|\mu_{t,\gamma_t}^*\| \le \frac{ \bfV_t(\tilde
\xi_t)-d(\tilde
\mu_t)}{{\underset{j=1,\ldots,p_t+2n}{\min}\{[\Gamma_t{\alpha_t}]\}}}.
\end{align*}
\end{proofs}   
\section{Proof of Theorem~\ref{th:theorem2}}
\label{app:proof_theorem_2}
This section contains the proof of Theorem~\ref{th:theorem2} presented in Section~\ref{sec:closed_loop_certification}.
\begin{thm}
Assuming that there exist $\epsilon_t$ ($\timeint{0}{N}$) and $\epsilon_{z_t}$
($\timeint{1}{N}$) selected according to Lemma~\ref{lem:improve_epsilon_t}, then the following holds:
\vspace{-0.5cm}
\begin{equation*}
\mathcal V^*\le\bfV_{\gamma}\le\mathcal
 V^*+2\sum_{t=0}^N\mathcal R_t\sqrt{p_t}\bar
\gamma_t,
\end{equation*}
where $\bar \gamma_t:=\epsilon_t+\maxC{t}\alpha_t$.
\end{thm}
\begin{proofs}
Due to the tightening of the original inequality constraints, $
\bfV_{\gamma}\geq  \mathcal V^*$,
given that, as proved in Section~\ref{subsec:improve_tightening}, the feasible region of the tightened subproblems is inside the one of the original subproblems. 

Recall that the consolidated prediction satisfies the equality
constraints~\eqref{eq:dyn} by construction. Hence, the following holds:
\begin{align*}
\bfV_{\gamma}\!\!
\leq\!\! & \sum\limits_{t=0}^N \big[\max\limits_{\lambda\geq
0}(\min\limits_{x_t,u_t}\mathcal V_t(x_t,u_t) + \left\langle\lambda,C_tx_t+D_tu_t+g_t
\right\rangle)\\
&\quad\quad +\left\langle
\lambda_{\gamma_t}^*,[I_{{p_t}}~0]\gamma_t\right\rangle\big]\\
\leq & \mathcal V^* + 2\sum\limits_{t=0}^N\mathcal R_t\sqrt{p_t}(\epsilon_t+\!\!\!\maxC{t}\alpha_t).
\end{align*}
where $[I_{{p_t}}~0]\gamma_t$ selects the first $p_t$ components of the vector
$\gamma_t$.
\end{proofs}
\section{Proof of Lemma~\ref{lem:update_y}}
\label{app:proof_lemma_update_y}
This section contains the proof of Lemma~\ref{lem:update_y} presented in Section~\ref{sec:closed_loop_certification}.
\begin{lem}
Let $\bar y_{\gamma}$ be defined as  $\bar y_{\gamma}:=[\bar
y_{0,\gamma_0}^T\ldots\bar y_{N,\gamma_N}^T]=[(x_{0}^T~\bar
u_{0,\gamma_0}^T)\ldots(\bar x_{N-1,\gamma_{N-1}}^T~\bar u_{N-1,\gamma_{N-1}}^T)~(\bar
x_{N,\gamma_N}^T))]^T$. Then, $\tilde y^+$ at the next problem instance, is given
by:
\begin{align*}
\tilde y^+ & = [\bar y_{{\gamma}_{[2:N+1]}}((A+BK_f)\bar x_{N,\gamma_N})^T]^T 
\end{align*}  
\end{lem} 
\begin{proofs}
We can use a similar argument as the one of Lemma 4.2 in~\cite{NecoaraOCAM},
 recalling that $C_t=C$ for $\timeint{0}{N-1}$, $C_t=F_N$ for $t=N$, $D_t=D$ 
 for $\timeint{0}{N-1}$, $D_t=0$ for $t=N$, and $\mathcal X_N\subseteq\mathcal X$ 
 according to Assumption~\ref{ass:terminalSet}.
\end{proofs}
\fi
\ifdraft
 
\else
 
 \fi
\end{document}